\documentclass[a4paper,reqno]{amsart}

\usepackage{amsmath}
\usepackage{paralist}
\usepackage{graphics}
\usepackage{epsfig}
\usepackage{amsfonts}
\usepackage{amssymb}
\usepackage{hyperref}
\usepackage{epstopdf}
\usepackage{color}

\newtheorem{theorem}{Theorem}[section]
\newtheorem{lemma}[theorem]{Lemma}

\def\ifl{\iffalse }

\def\bc{\begin{center}}       \def\ec{\end{center}}
\def\ba{\begin{array}}        \def\ea{\end{array}}
\def\be{\begin{equation}}     \def\ee{\end{equation}}
\def\bea{\begin{eqnarray}}    \def\eea{\end{eqnarray}}
\def\beaa{\begin{eqnarray*}}  \def\eeaa{\end{eqnarray*}}

\numberwithin{equation}{section}

\newtheorem{definition}[theorem]{Definition}

\newtheorem{remark}[theorem]{Remark}
\numberwithin{equation}{section}

\begin{document}

\title[Two-species chemotaxis model with signal production]
{Global solvability in a two-species chemotaxis system with signal production}

\author{Guoqiang Ren}
\address{School of Mathematics and Statistics, Huazhong University of Science and Technology, Wuhan, 430074, Hubei, P. R. China;  Hubei Key Laboratory of Engineering Modeling and Scientific Computing,
       Huazhong University of Science and Technology,
       Wuhan, 430074, Hubei, P. R. China}
\email{597746385@qq.com.},

\author{Tian Xiang$^*$}
\address{Institute for Mathematical Sciences, Renmin University of China, Bejing, 100872, China}
\email{txiang@ruc.edu.cn}
\thanks{$^*$ Corresponding author.}




\begin{abstract}

In this paper, we study  the following prototypical two-species chemotaxis system with Lotka-Volterra competition and signal production:
$$\left\{ \begin{array}{lll}
u_t=\Delta u-\nabla\cdot(u\nabla w)+u(1-u^{\theta-1}-v),  &x\in \Omega, t>0, \\[0.2cm]
v_t=\Delta v-\nabla\cdot(v\nabla w)+v(1-v-u), & x\in \Omega, t>0, \\[0.2cm]
w_t=\Delta w-w+ u+ v,  & x\in \Omega, t>0.
  \end{array}\right.(\ast)
$$
We show that if $\theta>1+\frac{N-2}{N}$, the associated Neumann initial-boundary value problem for   ($\ast$) admits a global generalized solutions in a bounded and smooth domain $\Omega\subset\mathbb{R}^N$ $(N\geq2)$ under merely integrable initial data.
 \end{abstract}

\maketitle

\section{Introduction and statement of main result}

In this work, we consider a homogenous  Neumman initial-boundary value problem (IBVP) for the following  two-species chemotaxis system with Lotka-Volterra type competitive kinetics and signal production:
\be \label{1.1}\begin{cases}
u_t=\Delta u-\chi_1\nabla\cdot(u\nabla w)+\mu_1u(1-u^{\theta-1}-a_1v),  &x\in \Omega, t>0, \\[0.2cm]
v_t=\Delta v-\chi_2\nabla\cdot(v\nabla w)+\mu_2v(1-v-a_2u), & x\in \Omega, t>0, \\[0.2cm]
w_t=\Delta w-w+\alpha u+\beta v,  & x\in \Omega, t>0, \\[0.2cm]
\frac{\partial u}{\partial \nu}=\frac{\partial v}{\partial \nu}=\frac{\partial w}{\partial \nu}=0, & x\in \partial \Omega, t>0,\\[0.2cm]
u(x,0)=u_0(x), \  v(x,0)=v_0(x), \   w(x,0)=w_0(x), \  & x\in \Omega,  \end{cases}  \ee
where $\Omega\subset\mathbb{R}^N$ $(N\geq2)$ is a bounded domain with smooth boundary $\partial\Omega$ and $\frac{\partial}{\partial\nu}$ denotes the outward normal derivative on  $\partial\Omega$, $u$ and $v$ represent the population densities of two species, $w$ denotes the concentration of the chemical, their given initial data $(u_0, v_0, w_0)$ are nonnegative typically with strong regularity. The model parameters $\chi_1,\chi_2, \mu_1,\mu_2,\alpha,\beta$ are positive constants, $a_1,a_2$ are nonnegative constants, and  the exponent $\theta>1$.

For  extensive progress on the single species (either $u\equiv 0$ or $v\equiv 0$) counterpart of \eqref{1.1}, we refer the interested reader to  \cite{La15-JDE, VG-17-NA, 42, Win21-ANS, Xiang18-SIAM} for  an incomplete selection of studies concerned with dynamical properties of classical and generalized solutions.

When $\theta=2$, Bai and Winkler \cite{16} proved that the IBVP  \eqref{1.1} admits a unique global bounded solution for all regular nonnegative initial data in the two-dimensional case. Moreover, when $a_1,a_2\in(0,1)$ and $\mu_1,\mu_2$ are sufficiently large, any global bounded solution is shown to  converges exponentially to the positive equilibrium  $(\frac{1-a_1}{1-a_1a_2},\frac{1- a_2}{1-a_1a_2},\frac{\alpha (1-a_1)+\beta (1-a_2)}{1-a_1a_2})$ as $t\rightarrow\infty$; if $a_1>1>a_2>0$ and $\mu_2$ is sufficiently large, then any global bounded solution converges exponentially  to $(0,1,\beta)$ as $t\rightarrow\infty$; if $a_1=1>a_2>0$ and $\mu_2$ is sufficiently large, then any global bounded solution  converges algebraically to $(0,1,\beta)$ as $t\rightarrow\infty$. Mizukami \cite{17} extended this result and showed that the system \eqref{1.1} admits a unique global uniformly bounded classical solution if $-w+\alpha u+\beta v$ is replaced by $h(u,v,w)$ and $\chi_1,\chi_2$ are replaced by $\chi_i(w)$, $i=1,2$, respectively, and $h, \chi_i$ $(i=1,2)$ fulfill appropriate conditions. In the three-dimensional case, global boundedness of classical solutions to  the IBVP \eqref{1.1} with $\theta=2$ was shown by  Lin and Mu \cite{18}  provided
 $$
\mu_1>27\chi_1+23\chi_2, \ \ \  \mu_2>27\chi_2+23\chi_1.
 $$
 The same boundedness was also derived by Li and Wang \cite{19} under
  $$
 \frac{\mu_1}{\alpha}>\frac{(9\chi^2_1+3\chi^2_2)(\sqrt{5}+\sqrt{2})}{2\sqrt{\chi^2_1+\chi^2_2}}, \ \
 \frac{\mu_2} {\beta}>\frac{(9\chi^2_2+3\chi^2_1)(\sqrt{5}+\sqrt{2})}{2\sqrt{\chi^2_1+\chi^2_2}}.
 $$
In higher dimensional cases, when  $\theta=2$, $\chi_1,\chi_2$ are respectively replaced by $\chi_i(w)$ $(i=1,2)$  and $\chi_i(w)\leq\frac{K_i}{(1+ \alpha_iw)^{k_i}}$ with some positive constants $K_i,\alpha_i$ and $k_i>1$ $(i=1,2)$, global boundedness of classical solutions was studied by Zhang and Li  \cite{20, 21} for $\mu_1$ and $\mu_2$ are suitably  large, and, global existence of weak solution was also shown for  $\mu_1>0$ and $\mu_2>0$ in convex domains. Very recently, when $\theta=2$, Jin et al. \cite{22} first extended the  boundedness  for  model \eqref{1.1} to the one with so-called  density dependent motility and nonlinear chemosensitivity; under further restrictions, convergence of bounded solutions to constant steady states was also shown like  \cite{16}.  When $\theta=2$ and the third equation simplifies to an elliptic equation, Lin et al. \cite{23} showed global boundedness and convergence of bounded solutions to  \eqref{1.1} under some largeness of $\mu_1$ and $\mu_2$ in weak competition cases. This boundedness result was further improved by  Wang \cite{24}. For other types of two-species chemotaxis systems, such as two-species chemotaxis-competition system with two signals, two-species chemotaxis-competition system with/without  loop and two-species chemotaxis-Navier-Stokes system, we  refer  to  \cite{25, 26, 27, 28, TMQY-20, 12.1, 29, 14.1, 30} and the references therein.

Before we proceed to our main results, let us also mention a close cousin chemotaxis system with  Lotka-Volterra competition  and signal-consumption instead  of signal-production (with the same boundary and initial conditions  suspended):
\be \label{1.3}\begin{cases}
u_t=\Delta u-\chi_1\nabla\cdot(u \nabla w)+\mu_1u(1-u^{\theta-1}-a_1v),  &x\in \Omega, t>0, \\[0.2cm]
v_t=\Delta v-\chi_2 \nabla\cdot(v\nabla w)+\mu_2v(1-v-a_2u), & x\in \Omega, t>0, \\[0.2cm]
w_t=\Delta w-(\alpha u+\beta v)w,  & x\in \Omega, t>0.   \end{cases}  \ee
Being the same research line as that of \eqref{1.1}, this system and its special  cases have been studied widely in respective of global solvability  of classical and generalized solutions, boundedness and convergence to constant equilibrium. First, without any damping source, global existence and boundedness of classical solutions to the IBVP \eqref{1.3} in $n$D  were showed by Zhang and Tao \cite{3}  under
\be\label{ZTcond}
\max\left\{\chi_1, \  \chi_2\right\}\|w_0\|_{L^\infty(\Omega)}<\pi\sqrt{\frac{2}{N}}.
\ee
Under this smallness condition, long time dynamics  are  also naturally obtained:
\be\label{uvw-cov-zt}
 \lim_{t\rightarrow \infty} \left(\left\|u(\cdot, t)-\bar{u}_0\right\|_{L^\infty(\Omega)}+\left\|v(\cdot, t)-\bar{v}_0\right\|_{L^\infty(\Omega)}+\left\|w(\cdot, t)\right\|_{L^\infty(\Omega)}\right)=0.
 \ee
Recently, we remove any smallness condition like \eqref{ZTcond} to show   global boundedness of classical solution in 2D and global existence,  eventual smoothness and convergence of weak solutions like \eqref{uvw-cov-zt} in 3,4,5D \cite{RX21-pre}.

In the presence of standard Lotka-Volterra competition, i.e., $\theta=2$, global  boundedness of classical solutions to \eqref{1.3} is ensured (with possibly nonlinear chemosensitivity) under
 $$
 \text{either} \  \ \   N\leq 2 \ \ \ \text{or} \  \  \   \max\left\{\chi_1, \  \chi_2\right\}\|w_0\|_{L^\infty(\Omega)}<\frac{\pi}{\sqrt{N+1}}.
$$
Furthermore, such bounded classical solutions are known (\cite{8, JX19,  5, 28, 3}) to convergence  for large $\mu_1$ and $\mu_2$ according to
\be\label{conv-log}
\left(u(\cdot, t), v(\cdot, t), w(\cdot, t)\right)\overset{\text{ in } L^\infty(\Omega)}\longrightarrow \begin{cases}  \left(\frac{1-a_1}{1-a_1a_2},  \frac{1-a_2}{1-a_1a_2}, 0\right),  &\text{if } a_1, a_2\in(0,1), \\[0.2cm]
\left(0,1, 0\right), &\text{if } a_1\geq 1> a_2>0, \\[0.2cm]
\left(1, 0,  0\right), &\text{if } 0<a_1< 1\leq a_2.
\end{cases}
\ee
By relaxing parameter restriction as in \cite{2,3}, recently,  global existence of generalized weak solutions and their long time  behaviors (similar to \eqref{conv-log}) to \eqref{1.3} in nD are shown by Ren and Liu \cite{4} under
$$
   \max\left\{\chi_1, \  \chi_2\right\}\|w_0\|_{L^\infty(\Omega)}<\frac{1}{2}.
$$
With generalized Lotka-Volterra competition, i.e., $\theta>1$ but not necessarily $\theta=2$, global existence of generalized solution to \eqref{1.3} is shown in  \cite{6} for  $\theta>1$. Furthermore, eventual smoothness and convergence of weak solutions after  some (perhaps long)  waiting time  was recently shown by Ren and Liu \cite{15} in the three-dimensional case.
When species diffusions depend on signal, global boundedness and convergence were shown by  Qiu et al. \cite{7}  in 2, 3D  for $\mu_1,\mu_2$  sufficiently large.  For related models, the interested reader can find a large body of literature \cite{10, 8, 8.1, 11, 9, 12, 29.1, 36, 13, 14}.

From the review above about signal production case, we find that global existence and convergence are examined only under the logistic damping, i.e.,   $\theta=2$. A natural question is  whether or not the IBVP  \eqref{1.1} allows  a global solution for $\theta<2$? In this project,  inspired by the recent  work about signal-absorption case \cite{6}, and also \cite{35, 34, 33}, we shall examine such a question in a general  framework  of global generalized solutions to the IBVP \eqref{1.1} in domains of $N\geq2$.  Even through only the third equation in \eqref{1.1} and \eqref{1.3} are different, we would add that the respective prototypical models of \eqref{1.1} (the KS chemotaxis-production model or the minimal chemotaxis model)  and \eqref{1.3} (the KS chemotaxis-consumption model) have totally different phenomena: the former is well-known to have blow-ups for $N=2$ and large mass or $N\geq 3$ \cite{42, Xiang18-SIAM}, while, the latter is globally and classically solvable for $N\leq2$ and  is weakly solvable for $N=3,4,5$ \cite{RX21-pre}.  On the other hand, we wish to use this opportunity to make previous related arguments more transparent, more smooth and more understandable. Hence, we initiate to solve  \eqref{1.3} globally in a generalized sense in this  project. To state our result more precisely, let us specify our working weak regularity of integrable initial data as follows:
\be \label{1.2}
\begin{cases}
u_0\geq 0, \ v_0\geq 0,\  w_0\geq 0, \  \ r=\max\left\{2,  \  \frac{N(2-\theta)}{2(\theta-1)}\right\}, \\[0.25cm]
(u_0, \ v_0, \ w_0)\in L^1(\Omega)\times L^1(\Omega)\times L^r(\Omega).
\end{cases}  \ee
Now, we state our main result about global existence of generalized solution \eqref{1.1}.
\begin{theorem}\label{Thm1.1}
Let $\Omega\subset\mathbb{R}^N$ $(N\geq2)$ be a bounded domain with smooth boundary, the initial data $(u_0,v_0,w_0)$ fulfill  \eqref{1.2}, the model parameters  $\chi_1, \chi_2, \mu_1, \mu_2, \alpha, \beta>0$, $a_1,a_2\geq0$ and
\be\label{theta-large}
\theta>1+\frac{N-2}{N}=\frac{2N-2}{N}.
\ee
Then there exists at least one triple $(u,v,w)$ of nonnegative functions satisfying
$$
  \begin{cases}
u\in L^\infty((0,\infty);L^1(\Omega)), \\[0.25cm]
v\in L^\infty((0,\infty);L^1(\Omega))\ \ \ \and \\[0.25cm]
w\in \bigcap_{p\in[1,\frac{N+2}{N+1})}L^p_{loc}([0,\infty);W^{1,p}(\Omega)),
\end{cases}
$$
which is a global generalized solution of \eqref{1.1} in the sense of Definition \ref{def2.1} below.
\end{theorem}
\begin{remark}[Notes on global existence of generalized solutions of \eqref{1.1}]
\
\
\
\begin{itemize}
\item[(i)] We require the space dimension $N\geq 2$ simply because the corresponding IVBP can be easily  globally and classically solvable for $N=1$ even without any damping sources. In the weak solution framework, we drop the parameter restrictions required in   \cite{19, 18, 20, 21}.
\item[(ii)] In the signal-consumption case, global existence of  generalized solution to \eqref{1.3} was ensured in \cite{6} for $\theta>1$. In the  signal-production case, we are only able to show such  result for $N=2$. The condition \eqref{theta-large} is mainly used in Lemma \ref{lem3.4},  and,  when $v\equiv0$, Lemma \ref{lem3.4} is automatically  true;  then, under  $\theta>1$ and $r=1$ in \eqref{1.2},  a small modification of our proof  yields the existence of global generalized solution, coinciding  with \cite{33}.
\end{itemize}
\end{remark}

In the sequel, the symbols $C_i$ and $c_i$ $(i=1,2,\cdots)$ will denote  generic positive constants which may vary line-by-line. For simplicity, $u(x,t)$ is written as $u$, the integral $\int_{\Omega}u(x)dx$ is written as $\int_{\Omega}u(x)$ and $\int^t_0\int_{\Omega}u(x)dxdt$ is written as $\int^t_0\int_{\Omega}u(x)$.

The contents of the remaining paper are outlined as follows. In Section 2, we first normalize the system \eqref{1.1} and then formulate a convenient approximating system \eqref{2.10}, which is globally and classically solvable (cf. Lemma \ref{lem2.1}), thereafter, we derive an important evolution  identity \eqref{2.16} associated with the approximating system, which facilitates the introduction of generalized solution (which motives beyond from integration by parts). In Section 3, we first provide  some fundamental $\varepsilon$-independent a-priori estimates for  the approximating system \eqref{2.10}, which allows us to employ compactness arguments to carry out the limiting procedure so as to derive a global generalized solution to \eqref{1.1} under the main condition \eqref{theta-large}, thus finishing the proof of Theorem \ref{Thm1.1}.

\section{Preliminaries and the concept of generalized solution}
In order to make our notion of generalized solution more intuitive and understandable (which,   motivated mainly from in \cite{33}, is based beyond integration by parts),  we  first introduce an approximating  system   for \eqref{1.3}, which is globally and classically solvable. To this end, for $\varepsilon\in(0,1)$ and $r$ given by \eqref{1.2}, we fix function families $(u_{0\varepsilon})_{\varepsilon\in(0,1)}\subset C^0(\overline{\Omega})$, $(v_{0\varepsilon})_{\varepsilon\in(0,1)}\subset C^0(\overline{\Omega})$ and $(w_{0\varepsilon})_{\varepsilon\in(0,1)}\subset W^{1,\infty}(\Omega)$ such that $u_{0\varepsilon},v_{0\varepsilon}$ and $w_{0\varepsilon}$ are nonnegative and
\be\label{l1-bdd-epsi}
\|u_{0\varepsilon}\|_{L^1}\leq 1+\|u_0\|_{L^1}, \ \  \|v_{0\varepsilon}\|_{L^1}\leq 1+\|v_0\|_{L^1}, \ \  \|w_{0\varepsilon}\|_{L^r}\leq 1+\|w_0\|_{L^r},
\ee
and, as $\varepsilon\searrow0$,  that
\be \label{2.11}\begin{cases}
u_{0\varepsilon}\rightarrow u_0\ \ \textmd{in}~L^1(\Omega)\ \ \textmd{and}\ \textmd{a.e.}\ \textmd{in}\ \Omega,\cr
v_{0\varepsilon}\rightarrow v_0\ \ \textmd{in}~L^1(\Omega)\ \ \textmd{and}\ \textmd{a.e.}\ \textmd{in}\ \Omega,\cr
w_{0\varepsilon}\rightarrow w_0\ \ \textmd{in}~L^r(\Omega)\ \ \textmd{and}\ \textmd{a.e.}\ \textmd{in}\ \Omega.
\end{cases}  \ee
In our forthcoming analysis, the involving model parameters only bring more inconvenient computations, play inessential roles and will not change our conclusion. Hence, except the key parameter $\theta$, we assume all other parameters to be unit:
\be\label{unit}
\chi_1=\chi_2=\mu_1=\mu_2=a_1=a_2=1.
\ee
Then, we are ready to  formulate a convenient approximating system of \eqref{1.1}:
\be \label{2.10}\begin{cases}
u_{\varepsilon t}=\Delta u_{\varepsilon}- \nabla\cdot(u_{\varepsilon}\nabla w_{\varepsilon})+ u_{\varepsilon}(1-u_{\varepsilon}^{\theta-1}- v_{\varepsilon}),  &x\in \Omega, t>0, \\[0.2cm]
v_{\varepsilon t}=\Delta v_{\varepsilon}- \nabla\cdot(v_{\varepsilon}\nabla w_{\varepsilon})+ v_{\varepsilon}(1-v_{\varepsilon}- u_{\varepsilon}), & x\in \Omega, t>0, \\[0.2cm]
w_{\varepsilon t}=\Delta w_{\varepsilon}-w_{\varepsilon}+\frac{u_{\varepsilon}+ v_{\varepsilon}}{1+\varepsilon( u_{\varepsilon}+v_{\varepsilon})},  & x\in \Omega, t>0, \\[0.2cm]
\frac{\partial u_{\varepsilon}}{\partial \nu}=\frac{\partial v_{\varepsilon}}{\partial \nu}=\frac{\partial w_{\varepsilon}}{\partial \nu}=0, & x\in \partial \Omega, t>0,\\[0.2cm]
u_{\varepsilon}(x,0)=u_{0\varepsilon}(x), \  v_{\varepsilon}(x,0)=v_{0\varepsilon}(x), \   w_{\varepsilon}(x,0)=w_{0\varepsilon}(x), \  & x\in \Omega.   \end{cases}  \ee
 By well-known local well-posed in chemotaxis related systems, the IBVP \eqref{2.10} admits a  unique  nonnegative and classical solution $(u_\varepsilon,v_\varepsilon,w_\varepsilon)$ defined on $\overline{\Omega}\times (0, T_{\varepsilon,max})$.  Now, applying  the widely known smoothing $L^p$-$L^q$-estimates of the
Neumann heat semigroup $(e^{t\Delta})_{t\geq 0}$ in $\Omega$ (cf.  \cite{42,Xiang18-SIAM})  to the third equation in \eqref{2.10}, one can easily see that $\|\nabla w_\varepsilon(\cdot,t)\|_{L^\infty}$ is uniformly bounded on $(0, T_{\varepsilon,max})$.  With this key estimate, one can easily use semigroup or energy estimate to show that both $\| u_\varepsilon(\cdot,t)\|_{L^\infty}$ and $\| v_\varepsilon(\cdot,t)\|_{L^\infty}$ are  uniformly bounded on $(0, T_{\varepsilon,max})$. These enable one to conclude first global existence ($T_{\varepsilon,max}=\infty$) and then boundedness, as stated in the following lemma.
\begin{lemma}\label{lem2.1}
Let  $\theta>1$ and $\Omega\subset\mathbb{R}^N$ $(N\geq1)$ be a bounded domain with smooth boundary and let the initial data $(u_{0\varepsilon}, v_{0\varepsilon}, w_{0\varepsilon})$ be as described above.  Then for each $\varepsilon\in (0,1)$, there exist a triple of nonnegative functions
$$
\left\{ \begin{array}{lll}
u_{\varepsilon}\in C^0(\overline{\Omega}\times(0,\infty))\cap C^{2,1}(\overline{\Omega}\times(0,\infty)),\cr
v_{\varepsilon}\in C^0(\overline{\Omega}\times(0,\infty))\cap C^{2,1}(\overline{\Omega}\times(0,\infty)),\cr
w_{\varepsilon}\in\bigcap\limits_{q>N}C^0([0,\infty);W^{1,q}(\Omega))\cap C^{2,1}(\overline{\Omega}\times(0,\infty)),
\end{array}\right.
$$
and that $(u_{\varepsilon},v_{\varepsilon},w_{\varepsilon})$ solves \eqref{2.10} classically. Furthermore, for $t>0$,
\be\label{2.12}
\int_{\Omega}u_{\varepsilon}(\cdot,t)\leq m_1:=\max\left\{1+\| u_0\|_{L^1}, \ \  (\theta-1)\left(\frac{2}{\theta}\right)^\frac{\theta}{\theta-1}|\Omega|\right\},
\ee
\be\label{2.13}
\int_{\Omega}v_{\varepsilon}(\cdot,t)\leq m_2:=\max\left\{1+\| v_0\|_{L^1}, \ \   |\Omega|\right\}
\ee
and, for all $T>0$,
\be\label{2.14}
\int^T_0\int_{\Omega}u^\theta_{\varepsilon}\leq m_1T+1+\|u_0\|_{L^1},
\ee
\be\label{2.15}
\int^T_0\int_{\Omega}v^2_{\varepsilon}\leq m_2T+1+\|v_0\|_{L^1}.
\ee
Finally, if $\theta>\frac{2N-2}{N}$, then, for
\be\label{w-epsi-bdd}
\begin{cases}
1\leq p\leq \max\left\{2,  \   \frac{N(2-\theta)}{2(\theta-1)}\right\},  & \text{ if  } N=2,3, \\[0.25cm]
1\leq p\leq \max\left\{1,  \  \frac{N(2-\theta)}{2(\theta-1)}\right\},  &\text{ if  } N\geq 4,
\end{cases}  \ee
there exists $C_2=C_2(p)>0$ such that
\be\label{w-bdd}
\|w_\varepsilon(\cdot, t)\|_{L^p}\leq C_2, \ \  \ \forall t>0.
\ee
\end{lemma}
\begin{proof} As outlined above, the  global existence follows, see details in similar systems \cite{4, 39}. Now, integrating the first equation in \eqref{2.10} over $\Omega$ by parts and using the nonnegativity of $u_\varepsilon$ and  $v_\varepsilon$, we obtain
\be\label{b-uvw4*}
    \frac{d}{dt}\int_\Omega u_\varepsilon\leq\int_\Omega \left( u_\varepsilon-   u_\varepsilon^\theta\right)\leq -\int_\Omega   u_\varepsilon+ (\theta-1)\left(\frac{2}{\theta}\right)^\frac{\theta}{\theta-1}|\Omega|.
\ee
Solving this simple ODI (ordinary differential inequality) directly, we get
$$
\int_\Omega u_\varepsilon\leq \| u_{0\varepsilon}\|_{L^1}e^{- t}+ (\theta-1)\left(\frac{2}{\theta}\right)^\frac{\theta}{\theta-1}|\Omega|
\left(1-e^{- t}\right),
$$
which along with \eqref{l1-bdd-epsi} entails \eqref{2.12}. Now, an integration of \eqref{b-uvw4*} from $0$ to $T$ and a use of \eqref{2.12} imply \eqref{2.14}. The estimates for $v_\varepsilon$ in \eqref{2.13} and \eqref{2.15} can be easily obtained in the same way via setting $\theta=2$.

To derive an $\varepsilon$-independent bound for $w_\varepsilon$, we first use variation-of-constants formula to write  the $w_\varepsilon$-equation in  \eqref{2.10} as
$$
w_\varepsilon(\cdot, t)=e^{t(\Delta -1)}w_{0\varepsilon}+\int_0^te^{(t-s)(\Delta -1)}\frac{(u_{\varepsilon}+ v_{\varepsilon})(\cdot,s)}{1+\varepsilon( u_{\varepsilon}+ v_{\varepsilon})(\cdot,s)}ds.
$$
For $p$ satisfying \eqref{w-epsi-bdd}, we first clearly have $p\leq r$ by \eqref{1.2}, and then,  since $\theta>\frac{2N-2}{N}$ we see that $ p<\frac{N}{(N-2)^+}$ or equivalently $N-(N-2)p>0$. Then we apply  the well-known smoothing $L^p$-$L^q$ type estimates  of the
Neumann heat semigroup $(e^{t\Delta})_{t\geq 0}$  along with  \eqref{l1-bdd-epsi}, \eqref{2.12} and \eqref{2.13}  to deduce
\begin{equation*}
\begin{split}
\|w_\varepsilon(\cdot,t)\|_{L^p}&\leq e^{-t}\| e^{t\Delta}w_{0\varepsilon}\|_{L^p}+\int^t_0  \left\|e^{-(t-s)}e^{(t-s)\Delta}\frac{(u_{\varepsilon}+ v_{\varepsilon})(\cdot,s)}{1+\varepsilon( u_{\varepsilon}+ v_{\varepsilon})(\cdot,s)} \right\|_{L^p}ds\\
&\leq c_1\|w_{0\varepsilon}\|_{L^p}+c_1\int^t_0\left[1+(t-s)^{-\frac{N}{2}(1-\frac{1}{p})}\right]
e^{-(t-s)}\left\| (u_{\varepsilon}+ v_{\varepsilon})(\cdot,s)\right\|_{L^1}ds\\
& \leq c_2\|w_{0\varepsilon}\|_{L^r}+c_2\int^t_0\left[1+\tau^{-\frac{N}{2}(1-\frac{1}{p})}\right]
e^{-\tau}d\tau\\
&\leq c_2\left(1+\|w_0\|_{L^r}\right)+c_2\left(2+\frac{2p}{N-(N-2)p}\right), \  \  \ \forall t>0,
\end{split}
\end{equation*}
yielding immediately the desired   $\varepsilon$-independent bound \eqref{w-bdd}.
\end{proof}
Evidently, a formal limit of \eqref{2.10} as $\varepsilon\rightarrow 0+$ yields \eqref{1.1} with \eqref{unit}. In the sequel, we shall carry out rigourous limiting procedure to construct  a generalized solution to \eqref{1.1} out of \eqref{2.10}. To this purpose, we provide the following evolution  identity, which is an important  preparation for our concept of generalized solution.
\begin{lemma}
For $\varepsilon\in[0,1)$ and $T\in(0,\infty]$, let $(u,  v, w)$ from  $(C^{2,1}(\overline{\Omega}\times(0,T)))^3$  solve  \eqref{2.10}  on $\Omega\times(0,T)$. Then for any $C^2([0,\infty))$ functions $\phi,\xi$ and $\Phi$ with  $\phi\geq0$, $\xi>0$, $\phi''>0$ and $\Phi'=\sqrt{\phi''}$ on $[0,\infty)$ and for $\psi\in C^\infty(\overline{\Omega}\times(0,T))$, it holds that
\be\label{2.16}
\begin{split}
&\int_{\Omega}\partial_t\{\phi(u)\xi(w)\}\cdot\psi\cr
&=-\int_{\Omega}\Bigl|\nabla(\Phi(u)\sqrt{\xi(w)})+\Bigl\{\frac{\phi'(u)}{\sqrt{\phi''(u)}}\cdot
\frac{\xi'(w)}{\sqrt{\xi(w)}}-\frac{1}{2}\Phi(u)\frac{\xi'(w)}{\sqrt{\xi(w)}}\cr
&\ \ \ \ -\frac{1}{2}u\sqrt{\phi''(u)}\cdot\sqrt{\xi(w)}\Bigl\}\nabla w\Bigl|^2\psi\cr
&\ \ \ \ -\int_{\Omega}\Bigl\{\phi(u)\xi''(w)-\frac{\phi'^2(u)}{\phi''(u)}\cdot\frac{\xi'^2
(w)}{\xi(w)}-\frac{1}{4}u^2\phi''(u)\xi(w)\Bigl\}\cdot|\nabla w|^2\psi\cr
&\ \ \ \ -\int_{\Omega}\frac{\phi'(u)}{\sqrt{\phi''(u)}}\sqrt{\xi(w)}\nabla(\Phi(u)
\sqrt{\xi(w)})\cdot\nabla\psi\cr
&\ \ \ \ +\int_{\Omega}\Bigl\{ u\phi'(u)\xi(w) -\phi(u)\xi'(w)+ \frac{1}{2}\frac{\Phi(u)\phi'(u)}{\sqrt{\phi''(u)}}\xi'(w)\Bigl\}\nabla w\cdot\nabla\psi\cr
&\ \ \ \ +\int_{\Omega}\Bigl\{u(1-u^{\theta-1}-v)\phi'(u)\xi(w)
 +\left(\frac{u+ v}{1+\varepsilon ( u+ v)}-w\right)\phi(u)\xi'(w)\Bigl\}\cdot\psi.
\end{split} \ee
\end{lemma}
\begin{proof}
Honest computations via \eqref{2.10} and using the fact that
$$
\nabla u=\frac{1}{\sqrt{\phi''(u)\xi(w)}}\nabla
\left(\Phi(u)\sqrt{\xi(w)}\right)
-\frac{1}{2}\frac{\Phi(u)}{\sqrt{\phi''(u)}}\cdot
\frac{\xi'(w)}{\xi(w)}\nabla w,
$$
one can readily adapt  the computations in  \cite[Lemma 3.1]{33} or  \cite[Lemma 2.2]{6} to derive the evolution identity \eqref{2.16}.
\end{proof}

 We are now in the position to formulate  the notion of generalized solution to \eqref{1.1},  which largely motivates from the identity  \eqref{2.16} in the limiting case $\varepsilon =0$ and can  be found in \cite{33}, and also in \cite{6, 37}.
\begin{definition}\label{def2.1}
Let $(u,v,w)$ be a triple of nonnegative functions satisfying
\be \label{2.1}
\begin{cases}
\left(u, v, w\right)\in \left(L^1_{loc}(\overline{\Omega}\times[0,\infty)),  L^1_{loc}(\overline{\Omega}\times[0,\infty)),   L^1_{loc}([0,\infty);W^{1,1}(\Omega))\right), \\[0.25cm]
 u(1-u^{\theta-1}-v), \ \ \  v(1-v-u) \in L^1_{loc}(\overline{\Omega}\times[0,\infty)), \\[0.25cm]
 \nabla \ln(v+1), \ \ \ \frac{v}{v+1}\nabla w  \in L^2_{loc}(\overline{\Omega}\times[0,\infty); \mathbb{R}^N).
 \end{cases}
\ee
Then $(u,v,w)$ will be called a global generalized solution of \eqref{1.1} if
\be\label{2.2}
\begin{split}
&-\int^\infty_0\int_{\Omega}\ln(v+1)\psi_t-\int_{\Omega}\ln(v_0+1)\psi(\cdot,0)\cr
&\geq \int^\infty_0\int_{\Omega}|\nabla\ln(v+1)|^2\psi
-\int^\infty_0\int_{\Omega}\nabla\ln(v+1)\cdot\nabla\psi\\
&\ \ +\int^\infty_0\int_{\Omega}\frac{v}{v+1}\nabla w\cdot\nabla\psi-\int^\infty_0\int_{\Omega}\frac{v\psi}{v+1}\nabla w\cdot\nabla\ln(v+1)\\
&\ \ \ \ +\int^\infty_0\int_{\Omega}\frac{v\psi}{v+1}(1-v-u)
\end{split} \ee
and
 \be\label{2.3}
 \begin{split}
&-\int^\infty_0\int_{\Omega}w\psi_t-\int_{\Omega}w_0\psi(\cdot,0)\\
&=-\int^\infty_0\int_{\Omega}\nabla w\cdot\nabla\psi -\int^\infty_0\int_{\Omega}w\psi+\int^\infty_0\int_{\Omega}( u+ v)\psi
\end{split}
\ee
for all   nonnegative  $\psi\in C^\infty_0(\Omega\times[0,\infty))$, if
\be\label{2.4}
\int_{\Omega}u(\cdot,t)\leq\int_{\Omega}u_0+ \int^t_0\int_{\Omega}u(1-u^{\theta-1}- v)\ \ \textmd{for}\ \textmd{a.e.}\ t>0,
\ee
and if there exist functions $\left(\phi, \xi, \Phi\right)\in \left(C^2([0,\infty))\right)^3$ such that
\be\label{2.5}
\phi'<0,\ \ \xi>0\ \ \textmd{and}\ \ \phi''>0\ \ \textmd{on}\ [0,\infty),\ \  \Phi'=\sqrt{\phi''}\ \ \textmd{on}\ [0,\infty),
\ee
that
$$
\Bigl\{\phi(u)\xi''(w)-\frac{\phi'^2(u)}{\phi''(u)}\cdot\frac{\xi'^2(w)}{\xi(w)}-\frac{1}{4}u^2\phi''(u)\xi(w)\Bigl\}|\nabla w|^2,\ \ u\phi'(u)\xi(w)|\nabla w|,
$$
$$
\phi(u)\xi'(w)|\nabla w|, \ \ \frac{\Phi(u)\phi'(u)}{\sqrt{\phi''(u)}}\xi'(w)|\nabla w| \text{ and }  u(1-u^{\theta-1}-v)\phi'(u)\xi(w)\   \text{ as  well as}
$$
\be\label{2.7}
 u\phi(u)\xi'(w), v\phi(u)\xi'(w)   \text{ and }  w\phi(u)\xi'(w)\  \text{ belong to }\ L^1_{loc}(\overline{\Omega}\times[0,\infty))
\ee
that
\be\label{2.8}
\Phi(u)\sqrt{\xi(w)}\in L^2_{loc}([0,\infty);W^{1,2}(\Omega)),
\ee
and that for each nonnegative $\varphi\in C^\infty_0(\Omega\times[0,\infty))$, the following inequality holds:
\be\label{2.9}
\begin{split}
&-\int^\infty_0\int_{\Omega}\phi(u)\xi(w)\varphi_t-\int_{\Omega}\phi(u_0)\xi(w_0)\varphi(\cdot,0)\cr
&\leq-\int^\infty_0\int_{\Omega}\Bigl|\nabla(\Phi(u)\sqrt{\xi(w)})+\Bigl\{\frac{\phi'(u)}{\sqrt{\phi''(u)}}\cdot\frac{\xi'(w)}{
\sqrt{\xi(w)}}-\frac{1}{2}\Phi(u)\frac{\xi'(w)}{\sqrt{\xi(w)}}\cr
&\ \ \ \ -\frac{1}{2}u\sqrt{\phi''(u)}\cdot\sqrt{\xi(w)}\Bigl\}\nabla w\Bigl|^2\varphi\cr
&\ \ \ \ -\int^\infty_0\int_{\Omega}\Bigl\{\phi(u)\xi''(w)
-\frac{\phi'^2(u)}{\phi''(u)}\cdot\frac{\xi'^2(w)}{\xi(w)}-\frac{1}{4}u^2
\phi''(u)\xi(w)\Bigl\}\cdot|\nabla w|^2\varphi\cr
&\ \ \ \ -\int^\infty_0\int_{\Omega}\frac{\phi'(u)}{\sqrt{\phi''(u)}}\sqrt{\xi(w)}\nabla(\Phi(u)\sqrt{\xi(w)})\cdot\nabla\varphi\cr
&\ \ \ \ +\int^\infty_0\int_{\Omega}\Bigl\{ u\phi'(u)\xi(w)-\phi(u)\xi'(w)+\frac{1}{2}\frac{\Phi(u)\phi'(u)}{\sqrt{\phi''(u)}}\xi'(w)\Bigl\}\nabla w\cdot\nabla\varphi\cr
&\ \ +\int^\infty_0\int_{\Omega}\left\{u(1-u^{\theta-1}-v)\phi'(u)\xi(w)
 +\left( u+ v-w\right)\phi(u)\xi'(w)\right\}\cdot\varphi
\end{split} \ee
\end{definition}

\begin{remark}
Analogous to the reasoning of  \cite{35, 38}, multiplying the inequality $u_t\geq \Delta u- \nabla\cdot(u\nabla w)+ u(1-u^{\theta-1}-v)$ by $\frac{\psi}{w+1}$, the above definition is in accordance with classical solution if $(u,v,w)$ is a global generalized  solution in the above sense that additionally fulfills $(u,v,w)\in(C^0(\overline{\Omega}\times[0,\infty))\cap C^{2,1}(\overline{\Omega}\times(0,\infty)))^3$, then $(u,v,w)$ actually solves \eqref{1.1} classically on $\Omega\times(0,\infty)$.
\end{remark}

In the end of this section, for later purpose and   for convenience of reference,  we present the following well-known Gagliardo-Nirenberg interpolation inequality.
\begin{lemma} \label{lem2.5} (Gagliardo-Nirenberg interpolation inequality \cite{40, 41})
Let $0<q\leq p\leq\frac{2N}{N-2}$. There exists a positive constant $C_{GN}$ such that
$$
\|u\|_{L^p(\Omega)}\leq C_{GN}\left(\|\nabla u\|^a_{L^2(\Omega)}\|u\|^{1-a}_{L^q(\Omega)}+\|u\|_{L^q(\Omega)}\right),  \ \ \forall u\in W^{1,2}(\Omega)\cap L^q(\Omega),
$$
where  $a=\frac{\frac{1}{q}-\frac{1}{p}}{\frac{1}{q}-\frac{1}{2}+\frac{1}{N}}\in(0,1)$.
\end{lemma}
\section{Global existence of generalized solution: Proof of Theorem \ref{Thm1.1}}
In this section, we first derive $\varepsilon$-independent estimates for solutions to the approximate system \eqref{2.10}, and then we  carry out the limit procedure to obtain the global existence of generalized solution, finishing the proof  Theorem \ref{Thm1.1}.
\begin{lemma}\label{lem3.1}
For  $\theta>1$,  there exists $C_1>0$ such that, for all $T>0$,
\be\label{3.1}
 \int^T_0\int_{\Omega}|u_{\varepsilon}(1-u_{\varepsilon}^{\theta-1}- v_{\varepsilon})|+
\int^T_0\int_{\Omega}|v_{\varepsilon}(1-v_{\varepsilon}-u_{\varepsilon})|\leq C_1(1+T), \ \ \forall \varepsilon\in(0,1).
\ee
Moreover, we have
\be\label{3.3}
\sup_{\varepsilon\in(0,1)}\sup_{t>0}\int_{\Omega}u_\varepsilon(\cdot,t)+
\sup_{\varepsilon\in(0,1)}\sup_{t>0}\int_{\Omega}v_\varepsilon(\cdot,t)<+\infty.
\ee
\end{lemma}
\begin{proof}
The estimate \eqref{3.3} follows directly from \eqref{2.13} and \eqref{2.14}.
 To show \eqref{3.1}, for simplicity, we let $f_\varepsilon=u_\varepsilon(1-u^{\theta-1}_\varepsilon -v_\varepsilon)$ and decompose  $f_\varepsilon=f_{\varepsilon}^+-f_{\varepsilon }^-$ where
\be\label{f+-def}
f_{\varepsilon }^+:=\max\{f_\varepsilon,0\}=\begin{cases} u_\varepsilon(1-u^{\theta-1}_\varepsilon -v_\varepsilon), & u^{\theta-1}_\varepsilon +v_\varepsilon\leq 1, \\[0.25cm]
0, & u^{\theta-1}_\varepsilon +v_\varepsilon> 1,
\end{cases} \leq 1
\ee
and
$$
f_{\varepsilon }^-:=\max\{-f_\varepsilon,0\}=\begin{cases} 0, & u^{\theta-1}_\varepsilon +v_\varepsilon\leq 1, \\[0.25cm]
-u_\varepsilon(1-u^{\theta-1}_\varepsilon -v_\varepsilon), & u^{\theta-1}_\varepsilon +v_\varepsilon> 1.
\end{cases}
$$
Integrating the first equation in \eqref{2.10} over $\Omega$ by parts, we get
$$
\frac{d}{dt}\int_{\Omega}u_\varepsilon
=\int_{\Omega}f_{\varepsilon}(u_\varepsilon,v_\varepsilon),
$$
which upon being integrated from $0$ to $T$ along with \eqref{l1-bdd-epsi}  yields
\be\label{3.5}
-\int_0^T\int_{\Omega}f_{\varepsilon}(u_\varepsilon,v_\varepsilon)=\int_{\Omega} u_{0\varepsilon}-\int_{\Omega} u_\varepsilon(\cdot, T)\leq 1+ \int_{\Omega} u_0.
\ee
Then, using \eqref{f+-def}, \eqref{3.5} and the fact $|f_{\varepsilon}|=f_{\varepsilon}^++f_{\varepsilon}^-=2f^+-f$, we infer
\be\label{3.9}
\begin{split}
\int^T_0\int_{\Omega}|f_{\varepsilon}(u_\varepsilon,v_\varepsilon)|&= 2\int^T_0\int_{\Omega}f_{\varepsilon}^+(u_\varepsilon,v_\varepsilon)
-\int^T_0\int_{\Omega}f_\varepsilon(u_\varepsilon,v_\varepsilon)\cr
&\leq 2|\Omega|T+1+\int_{\Omega}u_0,
\end{split} \ee
yielding the estimate for $u_\varepsilon$ in \eqref{3.1}. The estimates for $v_\varepsilon$ in \eqref{3.1} follows in a similar way via setting $\theta=2$.
\end{proof}

\begin{lemma}\label{lem3.2} The global solution of  \eqref{2.10} satisfies, for all $T>0$, that
\be\label{3.10}
(u_\varepsilon)_{\varepsilon\in(0,1)}  \text{ and }  (v_\varepsilon)_{\varepsilon\in(0,1)} \text{ are   uniformly integrable over }  \Omega\times(0,T).
\ee
\end{lemma}
\begin{proof}
The uniform integrability in \eqref{3.10} is implied directly from \eqref{2.14} and \eqref{2.15}. Indeed, given $\eta>0$ and $T>0$, we put
$$
\delta=\left(\frac{\eta^\theta}{m_1T+1+\|u_0\|_{L^1}}\right)^\frac{1}{\theta-1}.
$$
 Then for any measurable set $E\subset \Omega\times (0, T)$ with $|E|<\delta$, since $\theta>1$ we use H\"{o}lder's inequality to derive  from \eqref{2.14} that
 $$
 \int\int_E u_\varepsilon\leq  \left( \int_0^T\int_\Omega u_\varepsilon^\theta\right)^\frac{1}{\theta}|E|^\frac{\theta-1}{\theta}
 <\left(m_1T+1+\|u_0\|_{L^1}\right)^\frac{1}{\theta}\delta^\frac{\theta-1}{\theta}=\eta,
 $$
 giving rise to the uniform integrability of $(u_\varepsilon)_{\varepsilon\in(0,1)}$. Similarly,  the uniform integrability of $(v_\varepsilon)_{\varepsilon\in(0,1)}$ follows from \eqref{2.15}. Alternatively, one can also deduce \eqref{3.10} from \eqref{3.1} of Lemma \ref{lem3.1}.
\end{proof}
\begin{lemma}\label{lem3.3}
Let $p\in[1,\frac{N+2}{N+1})$. Then for any $T>0$, $(w_\varepsilon)_{\varepsilon\in(0,1)}$ is relatively compact with respect to the strong topology in $ L^p((0,T);W^{1,p}(\Omega))$.
\end{lemma}
\begin{proof} The essential idea has been
illustrated in  \cite[Lemma 5.1]{33}.  We here alternatively show the short  proof. We proceed with no loss of
 generality that $p>1$, and then, we pick a constant $\eta$ fulfilling
\be\label{3.25}
\frac{1}{2}<\eta<\frac{N+2-Np}{2p}\Longrightarrow 1-\eta-\frac{N}{2}(1-\frac{1}{p})>0, \  1-\left[\eta+\frac{N}{2}(1-\frac{1}{p})\right]p>0.
\ee
Let $A$ be  the realization of $-\Delta+1$ under the homogeneous Neumann boundary condition in $L^p(\Omega)$. Then the fact $\eta>\frac{1}{2}$ shows that the domain $D(A^\eta)$ of its fractional power $A^\eta$ satisfies
 \be\label{fractional}
D(A^\eta) \text{ is compactly embedded into } W^{1,p}(\Omega).
\ee
For $T>0$, by the semigroup representation of  $w_\varepsilon$-equation in \eqref{2.10} together with the well-known smoothing properties of heat semigroup $(e^{-tA})_{t\geq0}$ (cf. \cite{42}), thanks to \eqref{3.25} and the $L^1$-boundedness information in \eqref{l1-bdd-epsi} and \eqref{3.3},  there exist positive constants $c_1=c_1(p)$ and  $c_2=c_2(p)$  such that, for  all $t\in(0,T)$ and each $\varepsilon\in(0,1)$,
\begin{equation*}
\begin{split}
\|A^\eta w_\varepsilon(\cdot,t)\|_{L^p}&= \Bigl\|A^\eta e^{-tA}w_{0\varepsilon}+\int^t_0A^\eta e^{-(t-s)A}\frac{(u_{\varepsilon}+ v_{\varepsilon})(\cdot,s)}{1+\varepsilon( u_{\varepsilon}+ v_{\varepsilon})(\cdot,s)}ds\Bigl\|_{L^p}\cr
&\leq c_1t^{-\eta-\frac{N}{2}(1-\frac{1}{p})}\|w_{0\varepsilon}\|_{L^1}\cr
&\   +c_1\int^t_0(t-s)^{-\eta-\frac{N}{2}(1-\frac{1}{p})}\Bigl\|\frac{(u_{\varepsilon}+ v_{\varepsilon})(\cdot,s)}{1+\varepsilon( u_{\varepsilon}+ v_{\varepsilon})(\cdot,s)}\Bigl\|_{L^1}ds\cr
&\leq c_2t^{-\eta-\frac{N}{2}(1-\frac{1}{p})}+c_2\int^t_0(t-s)^{-\eta-\frac{N}{2}(1-\frac{1}{p})}ds\cr
&\leq c_2t^{-\eta-\frac{N}{2}(1-\frac{1}{p})}+\frac{c_2}{1-\eta-\frac{N}{2}(1-\frac{1}{p})}
T^{1-\eta-\frac{N}{2}(1-\frac{1}{p})},
\end{split}
\end{equation*}
which upon being integrated from $0$ to $T$ and a use of \eqref{3.25} allows us to see
\be\label{3.27}
\begin{split}
&\int^T_0\|A^\eta w_\varepsilon(\cdot,t)\|^p_{L^p}\cr
&\leq\frac{2^{p-1}c^p_2T^{1-p[\eta+\frac{N}{2}(1-\frac{1}{p})]}}{1-p[\eta+\frac{N}{2}(1-\frac{1}{p})]}
+2^{p-1}T\left(\frac{c_2T^{1-\eta-\frac{N}{2}(1-\frac{1}{p})}}{1-\eta
-\frac{N}{2}(1-\frac{1}{p})}\right)^p, \ \ \forall \varepsilon\in(0,1).
\end{split} \ee
On the other hand, for  $\varphi\in C^\infty(\overline{\Omega})$ and $t>0$, we deduce from \eqref{2.10} that
\begin{equation*}
\begin{split}
\left|\int_{\Omega}w_{\varepsilon t}(\cdot,t)\varphi\right|&=\left|-\int_{\Omega}\nabla w_\varepsilon\cdot\nabla\varphi-\int_{\Omega} w_\varepsilon\varphi+\int_{\Omega}\frac{( u_{\varepsilon}+  v_{\varepsilon})}{1+\varepsilon(  u_{\varepsilon}+ v_{\varepsilon})}\varphi\right|\\[0.25cm]
&\leq\|\nabla w_\varepsilon\|_{L^p}\|\nabla\varphi\|_{L^{\frac{p}{p-1}}}+\left( \|u_\varepsilon\|_{L^1} + \|v_\varepsilon\|_{L^1}+\|w_\varepsilon\|_{L^1}\right)\|\varphi\|_{L^\infty}.
\end{split}
\end{equation*}
 Notice that our assumption on $p$ implies  $\frac{p}{p-1}>N+2$,  the Sobolev embedding shows that
$W^{1,\frac{p}{p-1}}(\Omega)\hookrightarrow L^\infty(\Omega)$. Thus, there exists $c_3=c_3(p)>0$ such that, for all $t>0$ and any $\varepsilon\in(0,1)$,
$$
\|w_{\varepsilon t}(\cdot,t)\|_{(W^{1,\frac{p}{p-1}})^\ast}\leq c_3\left(\|\nabla w_\varepsilon\|_{L^p}+\|w_\varepsilon\|_{L^1} +\|u_\varepsilon\|_{L^1}+\|v_\varepsilon\|_{L^1}\right).
$$
In light of   \eqref{3.3}, \eqref{fractional} and \eqref{3.27},  for every $T>0$, we infer
$$
(w_{\varepsilon t})_{\varepsilon\in(0,1)}\text{ is bounded in }\ L^p((0,T);(W^{1,\frac{p}{p-1}}(\Omega))^*),
$$
which along with  the Aubin-Lions lemma,  \eqref{fractional} and \eqref{3.27}  concludes the proof.
\end{proof}

\begin{lemma}\label{lem3.4}
Let $\theta>\frac{2N-2}{N}$. Then for all $T>0$ there exists $C_3>0$ such that
\be\label{vgradw-bdd}
\int^T_0\int_{\Omega}\frac{v_\varepsilon^2}{(1+v_\varepsilon)^2}|\nabla w_\varepsilon|^2\leq C_3(1+T), \ \ \  \ \forall \varepsilon\in(0,1)
\ee
and
\be\label{3.12}
\int^T_0\int_{\Omega}|\nabla \ln(1+v_{\varepsilon})|^2=\int^T_0\int_{\Omega}\frac{|\nabla v_{\varepsilon}|^2}{(v_{\varepsilon}+1)^2}\leq C_3(1+T), \ \ \forall \varepsilon\in(0,1).
\ee
\end{lemma}
\begin{proof}
 Multiplying the third equation in \eqref{2.10} by $w_{\varepsilon}$, integrating by parts and using the Young inequality, we have, for any $\eta>0$,
\be\label{3.14}
\begin{split}
&\frac{1}{2}\frac{d}{dt}\int_{\Omega}w_{\varepsilon}^2+\int_{\Omega}|\nabla w_\varepsilon|^2+\int_{\Omega}w_\varepsilon^2\\
&\leq \int_{\Omega} u_{\varepsilon} w_{\varepsilon}+\int_{\Omega} v_{\varepsilon} w_{\varepsilon}\\
&\leq \eta \int_{\Omega}  w_{\varepsilon}^\frac{\theta}{\theta-1}+\frac{1}{\theta}
\left(\frac{\theta-1}{\theta\eta}\right)^{\theta-1} \int_{\Omega} u_{\varepsilon}^\theta+\eta \int_{\Omega}  w_{\varepsilon}^2+\frac{1}{4\eta}\int_{\Omega}  v_{\varepsilon}^2.
\end{split} \ee
Now, we put
$$
p=\max\left\{1,  \ \  \frac{N(2-\theta)}{2(\theta-1)}\right\}.
$$
Then thanks to our assumption $\theta>\frac{2N-2}{N}(\geq \frac{2N}{N+2})$, we have
$$
  p<\min\left\{\frac{N}{(N-2)^+}, \ \ \frac{\theta}{\theta-1}\right\}, \ \ \  \frac{2N[\theta-(\theta-1)p]}{(\theta-1)[2N-(N-2)p]}\leq 2.
$$
With these information at hand, based on the $\varepsilon$-independent bound  \eqref{w-bdd},  we  infer from   the Gagliardo-Nirenberg inequality (cf. Lemma \ref{lem2.5}) and Young's inequality  that
\be\label{3.15}
\begin{split}
\|w_{\varepsilon}\|_{L^\frac{\theta}{\theta-1}}
^\frac{\theta}{\theta-1}&\leq c_1\|\nabla w_{\varepsilon}\|_{L^2}^\frac{2N[\theta-(\theta-1)p]}{(\theta-1)[2N-(N-2)p]}\| w_{\varepsilon}\|_{L^p}^\frac{[(N+2)\theta-2N]p}{(\theta-1)[2N-(N-2)p]}+c_1\| w_{\varepsilon}\|_{L^1}^\frac{\theta}{\theta-1}\\
&\leq c_2\|\nabla w_{\varepsilon}\|_{L^2}^\frac{2N[\theta-(\theta-1)p]}{(\theta-1)[2N-(N-2)p]}+c_1\| w_{\varepsilon}\|_{L^1}^\frac{\theta}{\theta-1}\\
&\leq c_3\|\nabla w_{\varepsilon}\|_{L^2}^2+c_3.
\end{split} \ee
In a similar manner, or, simply  applying \eqref{3.15} with $\theta=2$, we have
\be\label{gn2}
\int_{\Omega}  w_{\varepsilon}^2\leq c_4\|\nabla w_{\varepsilon}\|_{L^2}^2+c_4.
\ee
Substituting \eqref{3.15} and \eqref{gn2} into \eqref{3.14} and then fixing sufficiently small $\eta>0$, we conclude  that
\be\label{grad w}
\frac{d}{dt}\int_{\Omega}w_{\varepsilon}^2+\int_{\Omega}|\nabla w_\varepsilon|^2+\int_{\Omega}w_\varepsilon^2\leq c_5+c_5\int_{\Omega} u_{\varepsilon}^\theta+c_5\int_{\Omega}  v_{\varepsilon}^2.
 \ee
An integration of \eqref{grad w} from $0$ to $T$ and uses of \eqref{l1-bdd-epsi},  \eqref{2.14} and \eqref{2.15} entail that
\be\label{gradw-est}
\int^T_0\int_{\Omega}|\nabla w_\varepsilon|^2\leq c_6\left(T+ \int^T_0\int_{\Omega}u_{\varepsilon}^\theta+ \int^T_0 \int_{\Omega}v_{\varepsilon}^2+\int_{\Omega}w_{0\varepsilon}^2\right)\leq c_7(1+T).
\ee
This directly gives rise to the estimate \eqref{vgradw-bdd}.

Next, multiplying the second equation in \eqref{2.10} by $\frac{1}{v_{\varepsilon}+1}$, integrating by parts and using Young's inequality, we obtain that
\be\label{3.16}
\begin{split}
&\frac{d}{dt}\int_{\Omega}\ln(v_{\varepsilon}+1)\\
&= \int_{\Omega}\frac{|\nabla v_{\varepsilon}|^2}{(v_{\varepsilon}+1)^2}-\int_{\Omega}
\frac{v_{\varepsilon}}{(v_{\varepsilon}+1)^2}\nabla v_{\varepsilon}\nabla w_{\varepsilon}
+\int_{\Omega}\frac{v_{\varepsilon}}{v_{\varepsilon}+1}(1-v_{\varepsilon}-u_{\varepsilon})\cr
&\geq \frac{1}{2}\int_{\Omega}\frac{|\nabla v_{\varepsilon}|^2}{(v_{\varepsilon}+1)^2}-\frac{1}{2}\int_{\Omega} \frac{v^2_{\varepsilon}}{(v_{\varepsilon}+1)^2}|\nabla w_{\varepsilon}|^2-\int_{\Omega}\frac{u_{\varepsilon}v_{\varepsilon}}{v_{\varepsilon}+1}- \int_{\Omega}\frac{v^2_{\varepsilon}}{v_{\varepsilon}+1}\cr
&\geq \frac{1}{2}\int_{\Omega}\frac{|\nabla v_{\varepsilon}|^2}{(v_{\varepsilon}+1)^2}-\frac{1}{2}\int_{\Omega}|\nabla w_{\varepsilon}|^2- \int_{\Omega}u_{\varepsilon}- \int_{\Omega}v_{\varepsilon}.
\end{split} \ee
Integrating  \eqref{3.16} from $0$ to $T>0$ and employing the $\varepsilon$-independent bounds in  \eqref{2.12}, \eqref{2.13},  \eqref{gradw-est} and the simple fact $\ln (1+z)\leq z$ for $z\geq0$, we end up with
\begin{equation*}
\begin{split}
&\int^T_0\int_{\Omega}\frac{|\nabla v_{\varepsilon}|^2}{(v_{\varepsilon}+1)^2}\\
&\leq2\int_{\Omega}\ln\left(1+v_{\varepsilon}(\cdot,T)\right)
+\int_0^T\int_{\Omega}|\nabla w_{\varepsilon}|^2+2\int_0^T \int_{\Omega}u_{\varepsilon}+2\int_0^T \int_{\Omega}v_{\varepsilon}\\
&\leq 2\int_{\Omega} v_{\varepsilon}(\cdot,T)
+\int_0^T\int_{\Omega}|\nabla w_{\varepsilon}|^2+2\int_0^T \int_{\Omega}u_{\varepsilon}+2\int_0^T \int_{\Omega}v_{\varepsilon}\\
&\leq c_8(1+T),
\end{split}
\end{equation*}
which is our desired $\varepsilon$-independent bound \eqref{3.12}.
\end{proof}

To  prepare ingredients in Definition \ref{def2.1}, we proceed to compute the corresponding expressions arising in \eqref{2.5}, \eqref{2.7}, \eqref{2.8} and  \eqref{2.9}.
\begin{lemma}\label{lem3.5}
Let $p>0$ and $k>0$, and define
\be\label{3.17}
\phi(s):=(s+1)^{-p},\  \Phi(s):=-2\sqrt{\frac{p+1}{p}}(s+1)^{-\frac{p}{2}}, \   \xi(\tilde{s}):=e^{-k\tilde{s}},\   s,  \tilde{s}\geq0.
\ee
Then
$$
\Phi'(s)=\sqrt{\phi''(s)}, \ \  \forall \ s\geq0,
$$
also, for all $s\geq0$ and $\tilde{s}\geq0$,
\begin{equation*}
\begin{split}
&\frac{\phi'(s)}{\sqrt{\phi''(s)}}\cdot\frac{\xi'(\tilde{s})}{\sqrt{\xi(\tilde{s})}}
-\frac{1}{2}\Phi(s)\frac{\xi'(\tilde{s})}{\sqrt{\xi(\tilde{s})}}
-\frac{1}{2}s\sqrt{\phi''(s)}\cdot\sqrt{\xi(\tilde{s})}\cr
&=-\frac{2k+ p(p+1)\frac{s}{s+1}}{2\sqrt{p(p+1)}}(s+1)^{-\frac{p}{2}}e^{-\frac{k\tilde{s}}{2}}
\end{split} \end{equation*}
and
\begin{equation*}
\begin{split}
&\phi(s)\xi''(\tilde{s})-\frac{\phi'(s)}{\phi''(s)}\cdot\frac{\xi'^2(\tilde{s})}{\xi(\tilde{s})}
-\frac{1}{4}s^2\phi''(s)
\xi(\tilde{s})\\
&=\frac{4k^2- p(p+1)^2\frac{s^2}{(s+1)^2}}{4(p+1)}(s+1)^{-p}e^{-k\tilde{s}}
\end{split}
 \end{equation*}
as well as
$$
\frac{\phi'(s)}{\sqrt{\phi''(s)}}\sqrt{\xi(\tilde{s})}=(s+1)^{-\frac{p}{2}}e^{-\frac{k\tilde{s}}{2}}
$$
and
$$
s\phi'(s)\xi(\tilde{s})-\phi(s)\xi'(\tilde{s})
+\frac{1}{2}\frac{\Phi(s)\phi'(s)}{\sqrt{\phi''(s)}}\xi'(\tilde{s})=-p s(s+1)^{-p-1}e^{-k\tilde{s}}.
$$
\end{lemma}
\begin{proof}
The conclusion is derived by direct calculations, cf.  \cite[Lemma 6.1]{33}.
\end{proof}
\begin{lemma}\label{lem3.6}
For $\varepsilon\in[0,1)$, let $(u_\varepsilon,  v_\varepsilon, w_\varepsilon)$ be the global classical solution of   \eqref{2.10} and for  $p>0$ and $k>0$, let
 \be\label{z-def}
 z_\varepsilon=(u_\varepsilon+1)^{-p}e^{-kw_\varepsilon}.
 \ee
Then, for every $\varphi\in C^\infty(\overline{\Omega}\times(0,\infty))$, one has
\be\label{3.23}
\begin{split}
&\int_{\Omega}z_{\varepsilon t}\varphi+\frac{4(p+1)}{p}\int_{\Omega}\left|\nabla z_\varepsilon^\frac{1}{2}+\frac{2k+ p(p+1)\frac{u_{\varepsilon} }{u_{\varepsilon}+1}}{2\sqrt{p(p+1)}}z_\varepsilon^\frac{1}{2} \nabla w_{\varepsilon} \right|^2\varphi\\
&= -\int_{\Omega}\frac{4k^2-p(p+1)^2\frac{u_\varepsilon^2}{(u_\varepsilon+1)^2}}{4(p+1)}
z_\varepsilon|\nabla w_\varepsilon|^2\varphi -2\int_{\Omega}z_\varepsilon^\frac{1}{2}\nabla z_\varepsilon^\frac{1}{2}\cdot\nabla\varphi\cr
&\ \   -p\int_{\Omega}\frac{u_\varepsilon z_\varepsilon}{u_\varepsilon+1} \nabla w_\varepsilon\cdot\nabla\varphi-p\int_{\Omega}\frac{u_\varepsilon z_\varepsilon}{u_\varepsilon+1} (1-u_\varepsilon^{\theta-1}- v_\varepsilon)\varphi\\
 & \ \ +k\int_{\Omega}w_{\varepsilon}z_\varepsilon\varphi-k\int_{\Omega}\frac{( u_{\varepsilon}+ v_{\varepsilon})z_\varepsilon}{1+\varepsilon( u_{\varepsilon}+ v_{\varepsilon})}\varphi.
\end{split} \ee
\end{lemma}
\begin{proof}
By \eqref{z-def} and \eqref{3.17}, it follows that $z_\varepsilon=\phi(u_\varepsilon)\xi(w_\varepsilon)$. Then the  identity \eqref{3.23}  follows from the computations provided in \eqref{2.16}   and Lemma \ref{lem3.5}.
\end{proof}

\begin{lemma}\label{lem3.7}
Suppose that  $p>0$ and $k>0$ satisfy
\be\label{3.28}
k>\frac{\sqrt{p}(p+1)}{2}.
\ee
Then for all $T>0$,  there exists $C_4=C_4(p,k)>0$ such that
\be\label{3.29}
\int^T_0\int_{\Omega}\left|\nabla\left\{(u_\varepsilon+1)^{-\frac{p}{2}}e^{-\frac{kw_\varepsilon}{2}}
\right\}\right|^2\leq C_4 (1+T), \ \ \forall \varepsilon\in(0,1)
\ee
and
\be\label{3.30}
\int^T_0\int_{\Omega}(u_\varepsilon+1)^{-p}e^{-kw_\varepsilon}|\nabla w_\varepsilon|^2\leq C_4(1+T),  \ \ \forall \varepsilon\in(0,1).
\ee
\end{lemma}
\begin{proof}
First, in view of  \eqref{3.28}, it is easy to see that
$$
\frac{4k^2- p(p+1)^2\frac{u_\varepsilon^2}{(u_\varepsilon+1)^2}}{4(p+1)}\geq  \frac{4k^2- p(p+1)^2}{4(p+1)}:=c_1>0 \ \ \text{on }  \Omega\times(0,\infty).
$$
For  $z_\varepsilon$   given by   \eqref{z-def} and for $T>0$, taking  $\varphi\equiv1$ in \eqref{3.23} and integrating  in $t$ from $0$ to  $T$, we get
\be\label{3.31}
\begin{split}
&\int_{\Omega}z_\varepsilon(\cdot,T)+c_1\int^T_0\int_{\Omega}z_\varepsilon|\nabla w_\varepsilon|^2\cr
&\ \ \ \ +\frac{4(p+1)}{p}\int^T_0\int_{\Omega}\left|\nabla z_\varepsilon^\frac{1}{2}+\frac{2k+ p(p+1)\frac{u_\varepsilon}{u_\varepsilon+1}}{2\sqrt{p(p+1)}}
z_\varepsilon^\frac{1}{2}\nabla w_\varepsilon\right|^2\cr
&\leq \int_{\Omega}(u_{0\varepsilon}+1)^{-p}e^{-kw_{0\varepsilon}}-p\int_0^T\int_{\Omega}\frac{u_\varepsilon z_\varepsilon}{u_\varepsilon+1} (1-u_\varepsilon^{\theta-1}- v_\varepsilon)\\
 & \ \ +k\int_0^T\int_{\Omega}w_{\varepsilon}(u_\varepsilon+1)^{-p}e^{-kw_\varepsilon}-k\int_0^T\int_{\Omega}\frac{( u_{\varepsilon}+ v_{\varepsilon})z_\varepsilon}{1+\varepsilon( u_{\varepsilon}+ v_{\varepsilon})}\\
&\leq \left(1+\frac{T}{e}\right) |\Omega|+p\int_0^T\int_{\Omega}u_\varepsilon\left|1-u_\varepsilon^{\theta-1}- v_\varepsilon\right|\\
&\leq c_2(1+T),
\end{split} \ee
where we have used \eqref{3.1} and the facts $z_\varepsilon\leq 1$ and  $kse^{-ks}\leq e^{-1}$ for $s>0$.

We notice from the fact  $|a+b|^2\geq \frac{1}{2}a^2 -b^2$ for all $a,b \in \mathbb{R}$  that
\begin{equation*}
\begin{split}
&\left|\nabla z_\varepsilon^\frac{1}{2}+\frac{2k+ p(p+1)\frac{u_\varepsilon} {u_\varepsilon+1}}{2\sqrt{p(p+1)}}z_\varepsilon^\frac{1}{2}\nabla w_\varepsilon\right|^2\cr
&\geq\frac{1}{2}|\nabla z_\varepsilon^\frac{1}{2}|^2-\left(\frac{2k+ p(p+1)}{2\sqrt{p(p+1)}}\right)^2z_\varepsilon|\nabla w_\varepsilon|^2:=\frac{1}{2}|\nabla z_\varepsilon^\frac{1}{2}|^2-c_3 z_\varepsilon|\nabla w_\varepsilon|^2
\end{split}
\end{equation*}
Setting  $c_4:=\min\{\frac{4(p+1)}{p},\frac{c_1}{2c_3}\}$, we deduce, for all $T>0$ and $\varepsilon\in(0,1)$,  that
\be\label{3.32}
\begin{split}
&c_1\int^T_0\int_{\Omega}z_\varepsilon|\nabla w_\varepsilon|^2+\frac{4(p+1)}{p}\int^T_0\int_{\Omega}\left|\nabla z_\varepsilon^\frac{1}{2}+\frac{2k+ p(p+1)\frac{u_\varepsilon}{u_\varepsilon+1}}{2\sqrt{p(p+1)}}
z_\varepsilon^\frac{1}{2}\nabla w_\varepsilon\right|^2\\
&\geq c_1\int^T_0\int_{\Omega}z_\varepsilon|\nabla w_\varepsilon|^2+c_4\int^T_0\int_{\Omega}\left|\nabla z_\varepsilon^\frac{1}{2}+\frac{2k+ p(p+1)\frac{u_\varepsilon}{u_\varepsilon+1}}{2\sqrt{p(p+1)}}
z_\varepsilon^\frac{1}{2}\nabla w_\varepsilon\right|^2\\
&\geq \frac{c_4}{2}\int^T_0\int_{\Omega} |\nabla  z_\varepsilon^\frac{1}{2}|^2+\left(c_1-c_3c_4\right)\int^T_0\int_{\Omega}z_\varepsilon|\nabla w_\varepsilon|^2\\
&\geq \frac{c_4}{2}\int^T_0\int_{\Omega} |\nabla  z_\varepsilon^\frac{1}{2}|^2+\frac{c_1}{2}\int^T_0\int_{\Omega}z_\varepsilon|\nabla w_\varepsilon|^2,
\end{split} \ee
Inserting  \eqref{3.32} into \eqref{3.31} and recalling \eqref{z-def}, we  achieve   \eqref{3.29} and \eqref{3.30}.
\end{proof}

 Based on the testing identity \eqref{3.23} of Lemma \ref{lem3.6} and the estimates in Lemma \ref{lem3.7}, we obtain certain space-time regularity for $z_\varepsilon$ and others.
\begin{lemma}\label{lem3.8}
Let $p>0$ and $k>\frac{\sqrt{p}(p+1)}{2}$, and let $m\in\mathbb{N}$ be such that $m>\frac{N}{2}$. Then for all $T>0$ there exists $C_5=C_5(p,k,m)>0$ such that
\be\label{3.33}
\int^T_0\left\|\partial_t\left\{(u_\varepsilon(\cdot,t)+1)^{-p}
e^{-kw_\varepsilon(\cdot,t)}\right\}\right\|_{(W^{m,2}(\Omega)))^\ast}dt\leq C_5(1+T), \ \ \forall \varepsilon\in(0,1)
\ee
as well as
\be\label{3.34}
\int^T_0\left\|\partial_t\ln\left(v_{\varepsilon}(\cdot,t)+1\right)\right\|_{(W^{m,2}(\Omega))^{\ast}}dt\leq C_5(1+T), \ \ \forall \varepsilon\in(0,1)
\ee
and
\be\label{3.35}
\int^T_0\left\|w_{\varepsilon t}(\cdot,t)\right\|_{(W^{m,2}(\Omega))^{\ast}}dt\leq  C_5(1+T), \ \ \forall \varepsilon\in(0,1).
\ee
\end{lemma}
\begin{proof}
For $T>0$ and for $z_\varepsilon$   given by  \eqref{z-def}, it follows from Lemmas \ref{lem2.1}, \ref{lem3.1} and  \ref{lem3.7}  that, for each $ \varepsilon\in(0,1)$,
\be\label{3.36}
\int^T_0\int_\Omega \left(|\nabla z_\varepsilon^\frac{1}{2}|^2+z_\varepsilon|\nabla w_\varepsilon|^2+
+ u_{\varepsilon}+ v_{\varepsilon}+u_\varepsilon \left|1-u_\varepsilon^{\theta-1}- v_\varepsilon\right|\right)\leq c_1(1+T).
\ee
Taking $\varphi\in C^\infty(\overline{\Omega})$ and applying Lemma \ref{lem3.6}, we see that
\be\label{3.37}
\begin{split}
&\left|\int_{\Omega}z_{\varepsilon t}\varphi\right|\\
&\leq  \frac{8(p+1)}{p}\int_{\Omega}\left|\nabla z_\varepsilon^\frac{1}{2}\right|^2|\varphi|+\frac{[2k+ p(p+1)]^2}{2p^2}\int_{\Omega}z_\varepsilon |\nabla w_\varepsilon |^2|\varphi|\\
&\  +\frac{4k^2+p(p+1)^2}{4(p+1)}\int_{\Omega}
z_\varepsilon|\nabla w_\varepsilon|^2|\varphi|+2\int_{\Omega}\left|z_\varepsilon^\frac{1}{2}\nabla z_\varepsilon^\frac{1}{2}\cdot\nabla\varphi\right|\\
&\ \   +p\int_{\Omega}\frac{u_\varepsilon z_\varepsilon}{u_\varepsilon+1}\left| \nabla w_\varepsilon\cdot\nabla\varphi\right|+p\int_{\Omega}\frac{u_\varepsilon z_\varepsilon}{u_\varepsilon+1} \left|(1-u_\varepsilon^{\theta-1}- v_\varepsilon)\varphi\right|\\\
 & \ \ +k\int_{\Omega}w_{\varepsilon}(u_\varepsilon+1)^{-p}e^{-kw_\varepsilon}|\varphi|+k\int_{\Omega}\frac{( u_{\varepsilon}+ v_{\varepsilon})z_\varepsilon}{1+\varepsilon( u_{\varepsilon}+ v_{\varepsilon})}|\varphi|\\
 &\leq \frac{8(p+1)}{p}\|\varphi\|_{L^\infty}\int_{\Omega}\left|\nabla z_\varepsilon^\frac{1}{2}\right|^2+\left(2\|\nabla z_\varepsilon^\frac{1}{2}\|_{L^2}+p\|z_\varepsilon^\frac{1}{2}\nabla w_\varepsilon\|_{L^2}\right)\|\nabla \varphi\|_{L^2}\\
 & +\left(\frac{[2k+ p(p+1)]^2}{2p^2}+\frac{4k^2+p(p+1)^2}{4(p+1)}\right)\|\varphi\|_{L^\infty}\int_{\Omega}
z_\varepsilon|\nabla w_\varepsilon|^2\\
&+\left(\frac{|\Omega|}{e}+p\int_{\Omega}u_\varepsilon \left|1-u_\varepsilon^{\theta-1}- v_\varepsilon\right|+k\int_\Omega ( u_{\varepsilon}+ v_{\varepsilon})\right)\|\varphi\|_{L^\infty},
\end{split} \ee
where we used the facts   $z_\varepsilon\leq 1$ and  $kse^{-ks}\leq e^{-1}$ for $s>0$.

Notice from  $m>\frac{N}{2}$ that  $W^{m,2}(\Omega)$ is continuously embedded into  $L^\infty(\Omega)$, which allows us to  find $c_2=c_2(m)>0$ such that $\|f\|_{L^\infty(\Omega)}\leq c_2\|f\|_{W^{m,2}(\Omega)}$ for all $f\in C^\infty(\overline{\Omega})$. Hence, by H\"{o}lder  inequality, we derive from \eqref{3.37} that
\be\label{3.38}
\begin{split}
&\left\|z_{\varepsilon t}\right\|_{(W^{m,2}(\Omega))^\ast}\\
&\leq \frac{8(p+1)}{p}c_2\int_{\Omega}\left|\nabla z_\varepsilon^\frac{1}{2}\right|^2+\left(2\|\nabla z_\varepsilon^\frac{1}{2}\|_{L^2}+p\|z_\varepsilon^\frac{1}{2}\nabla w_\varepsilon\|_{L^2}\right)\\
 & +\left(\frac{[2k+ p(p+1)]^2}{2p^2}+\frac{4k^2+p(p+1)^2}{4(p+1)}\right)c_2\int_{\Omega}
z_\varepsilon|\nabla w_\varepsilon|^2\\
&+\left(\frac{|\Omega|}{e}+p\int_{\Omega}u_\varepsilon \left|1-u_\varepsilon^{\theta-1}- v_\varepsilon\right|+k\int_\Omega ( u_{\varepsilon}+ v_{\varepsilon})\right)c_2\\
&\leq \left[1+\frac{8(p+1)}{p}c_2\right]\int_{\Omega}\left|\nabla z_\varepsilon^\frac{1}{2}\right|^2+\frac{p}{2}\int_\Omega z_\varepsilon|\nabla w_\varepsilon|^2\\
 & +\left(\frac{[2k+ p(p+1)]^2}{2p^2}+\frac{4k^2+p(p+1)^2}{4(p+1)}\right)c_2\int_{\Omega}
z_\varepsilon|\nabla w_\varepsilon|^2\\
&+1+\frac{p}{2}+\left(\frac{|\Omega|}{e}+p\int_{\Omega}u_\varepsilon \left|1-u_\varepsilon^{\theta-1}- v_\varepsilon\right|+k\int_\Omega ( u_{\varepsilon}+ v_{\varepsilon})\right)c_2.
\end{split} \ee
Integrating \eqref{3.38} from $0$ to $T$ and applying  \eqref{3.36} and \eqref{z-def}, we end up \eqref{3.33}.

To derive \eqref{3.34}, we multiply  the second equation in \eqref{2.10} by $\frac{\varphi}{v_{\varepsilon}(\cdot,t)+1}$, integrate by parts and use Young's inequality to deduce, for all $t>0$ and $\varepsilon\in(0,1)$, that
\begin{eqnarray*}
&&\int_{\Omega}\partial_t\ln(v_{\varepsilon}(\cdot,t)+1)\varphi\\
&&=\int_{\Omega}\frac{\varphi}{v_{\varepsilon}+1}[\Delta v_{\varepsilon}- \nabla\cdot(v_{\varepsilon}\nabla w_{\varepsilon})+ v_{\varepsilon}(1-v_{\varepsilon}- u_{\varepsilon})]\\
&&=\int_{\Omega}\frac{|\nabla v_{\varepsilon}|^2}{(v_{\varepsilon}+1)^2}\varphi-\int_{\Omega}\frac{\nabla v_{\varepsilon}}{v_{\varepsilon}+1}\cdot\nabla\varphi-\int_{\Omega}\frac{v_{\varepsilon}\varphi} {v_{\varepsilon}+1}\nabla w_{\varepsilon}\cdot\frac{\nabla v_{\varepsilon}}{v_{\varepsilon}+1}\cr
&&\ \ \ \ +\int_{\Omega}\frac{v_{\varepsilon}}{v_{\varepsilon}+1}\nabla w_{\varepsilon}\cdot\nabla\varphi+\int_{\Omega}\frac{v_{\varepsilon}\varphi }{1+v_{\varepsilon}}\left(1-v_{\varepsilon}-u_{\varepsilon}\right)\cr
&&\leq\|\varphi\|_{L^{\infty}}\int_{\Omega}\frac{|\nabla v_{\varepsilon}|^2}{(v_{\varepsilon}+1)^2}+\left(\int_{\Omega}\frac{|\nabla v_{\varepsilon}|^2}{(v_{\varepsilon}+1)^2}+ \frac{1}{4}\right)\|\nabla\varphi\|_{L^2}\cr
&&\ \ \ \ +\left(\frac{1}{4}\int_{\Omega}\frac{|\nabla v_{\varepsilon}|^2}{(v_{\varepsilon}+1)^2}+\int_{\Omega}|\nabla w_{\varepsilon}|^2\right)
\|\varphi\|_{L^{\infty}}\cr
&&\ \ \ \ +\left(\int_{\Omega}|\nabla w_{\varepsilon}|^2+\frac{1}{4}\right)\|\nabla\varphi\|_{L^2}+\|\varphi\|_{L^\infty}\int_{\Omega}
v_{\varepsilon}\left|1-v_{\varepsilon}-u_{\varepsilon}\right|.
\end{eqnarray*}
 Thus,  since $m>\frac{N}{2}$,  by the continuous embedding  $W^{m,2}(\Omega)\hookrightarrow L^{\infty}(\Omega)$,  \eqref{3.1}, \eqref{3.12} and \eqref{gradw-est}, for $T>0$, we conclude that
\begin{equation*}
\begin{split}
&\int_0^T\|\partial_t\ln(v_{\varepsilon t}(\cdot,t)+1)\|_{(W^{m,2}(\Omega))^{\ast}}\\
&\leq c_3(T+\int_0^T\int_{\Omega}\left(\frac{|\nabla v_{\varepsilon}|^2}{(v_{\varepsilon}+1)^2}+ |\nabla w_{\varepsilon}|^2+v_{\varepsilon}\left|1-v_{\varepsilon}-u_{\varepsilon}\right|\right)\\
&\leq c_4(1+T), \ \ \forall \varepsilon\in(0,1),
\end{split}
\end{equation*}
yielding precisely \eqref{3.34}.

Finally, multiplying the third equation in \eqref{2.10} by $\varphi$ and using Young's inequality, we have that
\begin{eqnarray*}
\int_{\Omega}w_{\varepsilon t}(\cdot,t)\varphi
&=&-\int_{\Omega}\nabla w_{\varepsilon}\cdot\nabla\varphi-\int_{\Omega}w_{\varepsilon}\varphi+\int_{\Omega}\frac{( u_{\varepsilon}+ v_{\varepsilon})}{1+\varepsilon( u_{\varepsilon}+ v_{\varepsilon})}\varphi\\
&\leq&\left(\int_{\Omega}|\nabla w_{\varepsilon}|^2+\frac{1}{4}\right)\|\nabla\varphi\|_{L^2}+\int_\Omega \left(w_{\varepsilon} + u_{\varepsilon}+v_{\varepsilon}\right)\|\varphi\|_{L^{\infty}}.
\end{eqnarray*}
By the continuous  embedding  $W^{m,2}(\Omega)\hookrightarrow L^{\infty}(\Omega)$,  Lemma \ref{lem2.1} and \eqref{gradw-est}, we see, for $T>0$ and each $\varepsilon\in(0,1)$, that
\begin{equation*}
\begin{split}
\int_0^T\|w_{\varepsilon t}(\cdot,t)\|_{(W^{m,2}(\Omega))^{\ast}}&\leq c_5\left(T+\int_0^T\int_{\Omega}\left( u_{\varepsilon}+v_{\varepsilon}+w_{\varepsilon} +|\nabla w_{\varepsilon}|^2\right)\right)\\
&\ \leq c_6(1+T),
\end{split}
\end{equation*}
and so \eqref{3.35} follows.
\end{proof}

Now, we are preparing to extract a suitable sequence of number $\varepsilon$ along with the respective solutions approach a limit in appropriate topologies, which serves as a generalized solution to \eqref{1.1}.

\begin{lemma}\label{lem3.9}
Under the conditions of Theorem \ref{Thm1.1},  there exist $(\varepsilon_j)_{j\in\mathbb{N}}\subset(0,1)$ with $\varepsilon_j\searrow0$ as $j\rightarrow\infty$ and nonnegative functions
\be\label{3.390}
\left(u, \ v\right)\in \left(L^\infty((0,\infty);L^1(\Omega))\right)^2,  \ \
w\in \bigcap_{p\in[1,\frac{N+2}{N+1})}L^p_{loc}([0,\infty);W^{1,p}(\Omega))
\ee
such that, as $\varepsilon=\varepsilon_j\searrow0$,
\be\label{3.39}
u_{\varepsilon}\rightarrow u\ \ \textmd{in}\ L^1_{loc}(\overline{\Omega}\times[0,\infty))\ \ \textmd{and}\ \textmd{a.e.}\ \textmd{in}\ \Omega\times(0,\infty)
\ee
and
\be\label{3.40}
v_{\varepsilon}\rightarrow v\ \ \textmd{in}\ L^1_{loc}(\overline{\Omega}\times[0,\infty))\ \ \textmd{and}\ \textmd{a.e.}\ \textmd{in}\ \Omega\times(0,\infty),
\ee
\be\label{3.41}
\ln(v_{\varepsilon}+1)\rightharpoonup\ln(v+1)\ \ \textmd{in}\ L^2_{loc}([0,\infty);W^{1,2}(\Omega)) \textmd{and}\ \textmd{a.e.}\ \textmd{in}\ \Omega\times(0,\infty)
\ee
as well as, for $1\leq p<\frac{N+2}{N+1}$,
\be\label{3.42}
w_{\varepsilon}\rightarrow w\ \ \textmd{in}\ L^p_{loc}(\overline{\Omega}\times[0,\infty))\ \ \textmd{and}\ \textmd{a.e.}\ \textmd{in}\ \Omega\times(0,\infty),
\ee
\be\label{3.43}
\nabla w_{\varepsilon}\rightarrow\nabla w\ \ \textmd{in}\ L^p_{loc}(\overline{\Omega}\times[0,\infty))\ \ \textmd{and}\ \textmd{a.e.}\ \textmd{in}\ \Omega\times(0,\infty)
\ee
\be\label{3.44}
\frac{(u_{\varepsilon} +  v_{\varepsilon})}{1+\varepsilon(  u_{\varepsilon}+  v_{\varepsilon}) }\rightarrow  u+  v\ \ \text{in }\ L^1_{loc}(\overline{\Omega}\times[0,\infty)).
\ee
 Furthermore, the identities \eqref{2.1}, \eqref{2.2}, \eqref{2.3} and inequality \eqref{2.4} hold for all $\psi\in C^{\infty}_0(\overline{\Omega}\times[0, \infty))$ and, with $m_1, m_2$ defined by \eqref{2.12} and \eqref{2.13}, it follows
$$
\int_{\Omega}\left(u(\cdot,t)+v(\cdot,t)\right)\leq m_1+m_2\ \ \textmd{for}\ a.e.\ t>0.
$$
\end{lemma}
\begin{proof}
For convenience, we shall rewrite the definition of $z_\varepsilon$ via \eqref{z-def} here:
\be\label{z-def+}
 z_\varepsilon=(u_\varepsilon+1)^{-p}e^{-kw_\varepsilon},  \  \  \ p>0,\  \    k>\frac{\sqrt{p} (p+1)}{2}.
 \ee
Now, for  $m>\frac{N}{2}$,  Lemmas \ref{lem3.7} and  \ref{lem3.8} enables us to see, for all $T>0$,  that
\be\label{3.46}
\int^T_0\int_{\Omega}|\nabla z_\varepsilon^{\frac{1}{2}}|^2+\int^T_0\|z_{\varepsilon t}(\cdot,t)\|_{(W^{m,2}(\Omega))^\ast}dt\leq c_1(1+T), \ \  \varepsilon\in(0,1).
\ee
The fact  that $0<z_\varepsilon\leq1$ in $\Omega\times(0,\infty)$ for all $\varepsilon\in(0,1)$ further shows
$$
\int^T_0\int_{\Omega}|\nabla z_\varepsilon|^2=4\int^T_0\int_{\Omega}z_\varepsilon|\nabla z_\varepsilon^{\frac{1}{2}}|^2\leq4c_1, \ \ \forall \varepsilon\in(0,1),
$$
and thus,
$$
(z_\varepsilon)_{\varepsilon\in(0,1)}\ \ \text{is bounded in}\ L^2((0,T);W^{1,2}(\Omega))\ \ \text{for all }\ T>0.
$$
On the other hand, we also know from \eqref{3.46}  that
$$
(z_{\varepsilon t})_{\varepsilon\in(0,1)}\ \ \textmd{is bounded in}\ L^1((0,T);(W^{m,2}(\Omega))^\ast)\ \ \textmd{for all}\ T>0.
$$
In view of  Lemma \ref{lem3.3},  it follows that
$$
(w_\varepsilon)_{\varepsilon\in(0,1)} \text{ is relatively compact  in }  L^p((0,T);W^{1,p}(\Omega)) \text{ for } p\in \left[1, \frac{N+2}{N+1}\right).
$$
 Hence, these along with Aubin-Lions lemma \cite{43} enable us to find $(\varepsilon_j)_{j\in\mathbb{N}}\subset(0,1)$ with $\varepsilon_j\searrow0$ as $j\rightarrow\infty$ and nonnegative functions $z$ and $w$ on $\Omega\times(0,\infty)$ such that \eqref{3.42} and \eqref{3.43}  hold and that  $z_\varepsilon\rightarrow z$ a.e. in $\Omega\times(0,\infty)$ as $\varepsilon=\varepsilon_j\searrow0$. Thus,
$$
u_{\varepsilon}=(e^{kw_{\varepsilon}}z_{\varepsilon})^{-\frac{1}{p}}-1\rightarrow u=(e^{kw}z)^{-\frac{1}{p}}-1 \text{ a.e. in } \Omega\times(0,\infty)~\text{as } \varepsilon=\varepsilon_j\searrow0,
$$
so that \eqref{3.39} results from Lemma \ref{lem3.2} and the Vitali convergence theorem.

Observing $\ln^2(1+s)\leq 2e^{-1}s$ for $s\geq0$, from \eqref{2.13} of Lemma \ref{lem2.1} and \eqref{3.12} of Lemma \ref{lem3.4}, we infer that
$$
(\ln(1+v_\varepsilon))_{\varepsilon\in(0,1)}\ \ \text{is bounded in}\ L^2((0,T);W^{1,2}(\Omega))\ \ \text{for all }\ T>0.
$$
Also, from \eqref{3.34} of Lemma \ref{lem3.8},  we see that
$$
(\partial_t(\ln(1+v_\varepsilon))_{\varepsilon\in(0,1)}\ \ \textmd{is bounded in}\ L^1((0,T);(W^{m,2}(\Omega))^\ast)\ \ \textmd{for all}\ T>0.
$$
Then, as above, these along with  the uniform integrability of $v_\varepsilon$ in Lemma \ref{lem3.2} and Aubin-Lions lemma \cite{43} imply the convergence prosperities in \eqref{3.40} and \eqref{3.41}. The convergence \eqref{3.44} follows simply from \eqref{3.39} and \eqref{3.40}.

Combining  \eqref{2.12}, \eqref{2.13},  \eqref{3.1}, \eqref{vgradw-bdd}, \eqref{3.39},  \eqref{3.40}, \eqref{3.41},  \eqref{3.42}, \eqref{3.43}  and Fatou's lemma, one can readily derive  \eqref{2.1} and \eqref{3.390}.

To check \eqref{2.2}, for any nonnegative  $\psi\in C^\infty_0(\Omega\times[0,\infty))$, we take $T>0$ so that supp $\psi\subset \Omega\times (0,T)$ and then test the second equation in \eqref{2.10} against $\psi$ to see, for all $\varepsilon\in(0,1)$,  that
\be\label{2.2+}
\begin{split}
&-\int^\infty_0\int_{\Omega}\ln(v_\varepsilon+1)\psi_t-\int_{\Omega}
\ln(v_{0\varepsilon}+1)\psi(\cdot,0)\cr
&=\int^\infty_0\int_{\Omega}|\nabla\ln(v_\varepsilon+1)|^2\psi
-\int^\infty_0\int_{\Omega}\nabla\ln(v_\varepsilon+1)\cdot\nabla\psi\\
&\ \ +\int^\infty_0\int_{\Omega}\frac{v_\varepsilon}{v_\varepsilon+1}\nabla w_\varepsilon\cdot\nabla\psi-\int^\infty_0\int_{\Omega}\frac{v_\varepsilon\psi}{v_\varepsilon+1}
\nabla w_\varepsilon\cdot\nabla\ln(v_\varepsilon+1)\\
&\ \ \ \ +\int^\infty_0\int_{\Omega}\frac{v_\varepsilon\psi}{v_\varepsilon+1}(1-v_\varepsilon-u_\varepsilon).
\end{split} \ee
In light of the weak $L^2$-convergence in \eqref{3.41} and the lower semi-continuity  of $L^2$-norm with respect to weak convergence, we have
\be\label{2.2+2}
\liminf_{\varepsilon=\varepsilon_j\searrow0}\int^\infty_0\int_{\Omega}
|\nabla\ln(v_\varepsilon+1)|^2\psi\geq \int^\infty_0\int_{\Omega}|\nabla\ln(v+1)|^2\psi
\ee
and
\be\label{2.2+3}
\lim_{\varepsilon=\varepsilon_j\searrow0} \int^\infty_0\int_{\Omega}\nabla\ln(v_\varepsilon+1)\cdot\nabla\psi
=\int^\infty_0\int_{\Omega}\nabla\ln(v+1)\cdot\nabla\psi.
\ee
By   \eqref{vgradw-bdd}, \eqref{3.40} and  \eqref{3.43},  up to a further subsequence, we have
\be\label{2.2+4}
\lim_{\varepsilon=\varepsilon_j\searrow0} \int^\infty_0\int_{\Omega}\frac{v_\varepsilon}{v_\varepsilon+1}\nabla w_\varepsilon\cdot\nabla\psi
=\int^\infty_0\int_{\Omega}\frac{v}{v+1}\nabla w\cdot\nabla\psi.
\ee
Thanks to \eqref{vgradw-bdd} and \eqref{3.12}, we infer from H\"{o}lder inequality that
\begin{equation*}
\begin{split}
&\int^\infty_0\int_{\Omega}\left|\frac{v_\varepsilon}{v_\varepsilon+1}
\nabla w_\varepsilon\cdot\nabla\ln(v_\varepsilon+1)\right|\\
&\leq \left(\int^T_0\int_{\Omega}
\left|\frac{v_\varepsilon}{v_\varepsilon+1}
\nabla w_\varepsilon\right|^2\right)^\frac{1}{2}\left(\int^T_0\int_{\Omega}
\left|\nabla\ln(v_\varepsilon+1)\right|^2\right)^\frac{1}{2}\\
&\leq c_2(1+T), \ \ \ \forall \varepsilon\in(0,1),
\end{split}
\end{equation*}
which together with \eqref{3.40}, \eqref{3.41} and \eqref{3.43} implies (up to a further subsequence) that
\be\label{2.2+5}
\lim_{\varepsilon=\varepsilon_j\searrow0} \int^\infty_0\int_{\Omega} \frac{v_\varepsilon \psi}{v_\varepsilon+1}
\nabla w_\varepsilon\cdot\nabla\ln(v_\varepsilon+1)= \int^\infty_0\int_{\Omega} \frac{v \psi}{v+1}
\nabla w\cdot\nabla\ln(v+1).
\ee
By the convergence properties in \eqref{2.11}, \eqref{3.39}, \eqref{3.40}, \eqref{3.41} and the uniform integrability in \eqref{3.1}, one can easily derive that
\be\label{2.2+6}
\begin{cases}
\lim_{\varepsilon=\varepsilon_j\searrow0} \int^\infty_0\int_{\Omega}\ln(v_\varepsilon+1)\psi_t= \int^\infty_0\int_{\Omega}\ln(v+1)\psi_t,\\[0.25cm]
\lim_{\varepsilon=\varepsilon_j\searrow0} \int_{\Omega}
\ln(v_{0\varepsilon}+1)\psi(\cdot,0)=\int_{\Omega}
\ln(v_0+1)\psi(\cdot,0),\\[0.25cm]
\lim_{\varepsilon=\varepsilon_j\searrow0}\int^\infty_0\int_{\Omega}
\frac{v_\varepsilon\psi}{v_\varepsilon+1}(1-v_\varepsilon-u_\varepsilon)=\int^\infty_0\int_{\Omega}
\frac{v\psi}{v+1}(1-v-u).
\end{cases}
 \ee
 Substituting \eqref{2.2+2},  \eqref{2.2+3}, \eqref{2.2+4}, \eqref{2.2+5} and \eqref{2.2+6} into
 \eqref{2.2+}, we arrive at the inequality stated in  \eqref{2.2}.

In the similar manners, with the aid of those established convergence properties and Fatou's lemma, one can readily check other statements, cf.  \cite[Lemma 3.10]{4} and  \cite[Lemma 7.1]{33}, for instance.
\end{proof}

\begin{lemma}\label{lem3.10}
For $p, k, T>0$, both $\left((u_{\varepsilon}+1)^{-p}u_\varepsilon(1-u^{\theta-1}_\varepsilon -v_\varepsilon)e^{-kw_{\varepsilon}}\right)_{\varepsilon\in(0,1)}$ and $\left((u_{\varepsilon}+1)^{-p}w_\varepsilon e^{-kw_{\varepsilon}}\right)_{\varepsilon\in(0,1)}$ are uniformly integrable  over $\Omega\times(0,T)$.
\end{lemma}
\begin{proof}
The desired conclusions follows  from  \eqref{3.1} and \eqref{w-bdd}, see Lemma \ref{lem3.2}.
\end{proof}

Now, we combine the approximation features of $(u_\varepsilon, v_\varepsilon, w_\varepsilon)$ gathered thus far to finally verify the remaining parts, especially the crucial superposition property \eqref{2.9},  of Definition \ref{def2.1} as follows.
\begin{lemma}\label{lem3.11}
Let $p>0$ and $k>\frac{\sqrt{p}(p+1)}{2}$, and let $\phi$, $\Phi$ and $\xi$ be defined by  \eqref{3.17}. Then \eqref{2.4}--\eqref{2.8} hold, and \eqref{2.9} is fulfilled for every nonnegative $\varphi\in C^\infty_0(\overline{\Omega}\times[0,\infty))$.
\end{lemma}
\begin{proof}
With the $z_\varepsilon$ notation in \eqref{z-def+}, for  $T>0$, Lemma \ref{lem3.7} implies that
 \begin{equation*}
 \begin{split}
 &\left\|z_\varepsilon^\frac{1}{2} \right\|_{L^2((0,T);W^{1,2}(\Omega))}+\left\|z_\varepsilon^\frac{1}{2}\nabla w_{\varepsilon}\right\|_{L^2(\Omega\times(0,T))}+\left\|\frac{u_{\varepsilon} z_\varepsilon^\frac{1}{2} }{u_{\varepsilon}+1} \nabla w_{\varepsilon}\right\|_{L^2(\Omega\times(0,T))}\\
 &\leq c_1(1+T), \  \ \ \ \forall \varepsilon\in (0,1),
 \end{split}
 \end{equation*}
 which upon identifying $z_\varepsilon=\phi(u_\varepsilon)\xi(w_\varepsilon)$ in \eqref{z-def+}  directly shows \eqref{2.8}.

Then, with  $(\varepsilon_j)_{j\in\mathbb{N}}$ provided by  Lemma \ref{lem3.9},  we know, as $\varepsilon=\varepsilon_j\searrow0$ (in the sequel, all limits are taken in this way, we should not specify them any more), that
$$
z_\varepsilon^\frac{1}{2}\rightarrow z^\frac{1}{2} ~~\textmd{a.e.}~\textmd{in}~\Omega\times(0,T)~~\textmd{as}~\textmd{well}~\textmd{as}
$$
$$
z_\varepsilon^\frac{1}{2}\nabla w_{\varepsilon}\rightarrow z^\frac{1}{2}\nabla w~~\textmd{a.e.}~\textmd{in}~\Omega\times(0,T)~~\textmd{and}
$$
$$
\frac{u_{\varepsilon} z_\varepsilon^\frac{1}{2} }{u_{\varepsilon}+1} \nabla w_{\varepsilon}\rightarrow \frac{u z^\frac{1}{2} }{u+1} \nabla w~~\textmd{a.e.}~\textmd{in}~\Omega\times(0,T),
$$
which in conjunction with Egorov's theorem allow us to infer
\be\label{3.48}
\nabla z_\varepsilon^\frac{1}{2}\rightharpoonup\nabla z^\frac{1}{2}~~\textmd{in}~L^2(\Omega\times(0,T))
\ee
and
\be\label{3.49}
z_\varepsilon^\frac{1}{2}\nabla w_{\varepsilon}\rightharpoonup z^\frac{1}{2}\nabla w~~\textmd{in}~L^2(\Omega\times(0,T))
\ee
as well as
\be\label{3.50}
\frac{u_{\varepsilon} z_\varepsilon^\frac{1}{2} }{u_{\varepsilon}+1} \nabla w_{\varepsilon}\rightharpoonup \frac{u z^\frac{1}{2} }{u+1} \nabla w~~\textmd{in}~L^2(\Omega\times(0,T)).
\ee
These in particular imply (cf. \eqref{3.23}) that
\be\label{3.51}
\begin{split}
&\nabla z_\varepsilon^\frac{1}{2}+\frac{2k+ p(p+1)\frac{u_{\varepsilon} }{u_{\varepsilon}+1}}{2\sqrt{p(p+1)}}z_\varepsilon^\frac{1}{2} \nabla w_{\varepsilon} \\
 &\ \ \rightharpoonup \nabla z_\varepsilon^\frac{1}{2}+\frac{2k+ p(p+1)\frac{u }{u+1}}{2\sqrt{p(p+1)}}z^\frac{1}{2} \nabla w  \text{ in } L^2(\Omega\times(0,T)).
\end{split} \ee
 and, by the fact   $z_\varepsilon\rightarrow z$ a.e. in $\Omega\times(0,\infty)$ from Lemma \ref{lem3.9} and the dominated convergence theorem, we readily infer that
\be\label{3.53}
z_\varepsilon \rightarrow z \  \text{ and } \  z_\varepsilon^\frac{1}{2}\rightarrow z^\frac{1}{2}~~\textmd{in}~L^2(\Omega\times(0,T)),
\ee
and so  we conclude  that
\be\label{3.52}
\begin{split}
\frac{u_{\varepsilon}z_{\varepsilon}}{1+u_{\varepsilon}}\nabla w_{\varepsilon}&=z_{\varepsilon}^\frac{1}{2}\cdot \frac{u_{\varepsilon}z_{\varepsilon}^\frac{1}{2}}{1+u_{\varepsilon}}\nabla w_{\varepsilon} \\
&\rightharpoonup z^\frac{1}{2}\cdot \frac{uz^\frac{1}{2}}{1+u}\nabla w= \frac{uz}{1+u}\nabla w~~\textmd{in}~L^1(\Omega\times(0,T)).
\end{split} \ee
Also,  \eqref{3.48} along  with \eqref{3.53} entails  that
\be\label{3.54}
z_\varepsilon^\frac{1}{2}\nabla z_\varepsilon^\frac{1}{2}\rightharpoonup z^\frac{1}{2}\nabla z^\frac{1}{2} \text{ in } L^1(\Omega\times(0,T)).
\ee
By the convergence in Lemma \ref{lem3.9} and uniform integrability in Lemma \ref{lem3.10},  the Vitali convergence theorem asserts that
\be\label{3.55}
 \frac{u_{\varepsilon}z_\varepsilon}{1+u_\varepsilon}(1-u^{\theta-1}_\varepsilon -v_\varepsilon)\rightarrow \frac{u z}{1+u}(1-u^{\theta-1} -v)  \  \text{ in } L^1(\Omega\times(0,T)).
\ee
 Lemmas \ref{lem3.9} and \ref{lem3.10} again together with the Vitali convergence theorem shows
\be\label{3.56}
w_{\varepsilon}z_\varepsilon \rightarrow wz\ \ \textmd{in}~L^1(\Omega\times(0,T))
\ee
and
\be\label{3.57}
\frac{( u_{\varepsilon}+ v_{\varepsilon})z_\varepsilon}{1+\varepsilon( u_{\varepsilon}+ v_{\varepsilon})}\rightarrow (u+ v)z\ \ \ \textmd{in}~L^1(\Omega\times(0,T)).
\ee
 Finally, given nonnegative function  $\varphi\in C^\infty_0(\Omega\times[0,\infty))$, we take $T>0$ such that supp $\varphi\subset\Omega\times(0,T)$. Integrating \eqref{3.23} in  $t$, we derive, for all $\varepsilon\in(0,1)$,  that
 \begin{equation*}
\begin{split}
&\frac{4(p+1)}{p}\int_0^\infty\int_{\Omega}\left|\nabla z_\varepsilon^\frac{1}{2}+\frac{2k+ p(p+1)\frac{u_{\varepsilon} }{u_{\varepsilon}+1}}{2\sqrt{p(p+1)}}z_\varepsilon^\frac{1}{2} \nabla w_{\varepsilon} \right|^2\varphi\\
&\ \ +\int_0^\infty\int_{\Omega}\frac{4k^2-p(p+1)^2\frac{u_\varepsilon^2}{(u_\varepsilon+1)^2}}{4(p+1)}
z_\varepsilon|\nabla w_\varepsilon|^2\varphi \\
&= \int_\Omega z_\varepsilon(\cdot, 0)\varphi(\cdot, 0)+\int_0^\infty\int_{\Omega}  z_\varepsilon \varphi_t -2\int_0^\infty\int_{\Omega}z_\varepsilon^\frac{1}{2}\nabla z_\varepsilon^\frac{1}{2}\cdot\nabla\varphi\\
&\ \   -p\int_0^\infty\int_{\Omega}\frac{u_\varepsilon z_\varepsilon}{u_\varepsilon+1} \nabla w_\varepsilon\cdot\nabla\varphi-p\int_0^\infty\int_{\Omega}\frac{u_\varepsilon z_\varepsilon}{u_\varepsilon+1} (1-u_\varepsilon^{\theta-1}- v_\varepsilon)\varphi\\
 & \ \ +k\int_0^\infty\int_{\Omega}w_{\varepsilon}z_\varepsilon\varphi-k\int_0^\infty\int_{\Omega}\frac{( u_{\varepsilon}+ v_{\varepsilon})z_\varepsilon}{1+\varepsilon( u_{\varepsilon}+ v_{\varepsilon})}\varphi.
\end{split}
\end{equation*}
It follows easily from $k>\frac{\sqrt{p}(p+1)}{2}$  that $\frac{4k^2-p(p+1)^2\frac{u_\varepsilon^2}{(u_\varepsilon+1)^2}}{4(p+1)}>0$. Then, by the convergence properties in Lemma \ref{lem3.9} and Fatou's lemma,
\begin{eqnarray*}
&&\int_0^\infty\int_{\Omega}\frac{4k^2-p(p+1)^2\frac{u^2}{(u+1)^2}}{4(p+1)}
z|\nabla w|^2\varphi\cr
&&\leq\liminf_{\varepsilon=\varepsilon_j\searrow0}\int_0^\infty\int_{\Omega}\frac{4k^2-p(p+1)^2\frac{u_\varepsilon^2}{(u_\varepsilon+1)^2}}{4(p+1)}
z_\varepsilon|\nabla w_\varepsilon|^2\varphi.
\end{eqnarray*}
Then  combining   \eqref{3.51} and the lower semi-continuity  of $L^2$-norm with respect to weak convergence, as well as \eqref{3.53}, \eqref{3.52}, \eqref{3.54}, \eqref{3.55}, \eqref{3.56}, \eqref{3.57} and \eqref{2.11}, we finally infer that
\begin{equation*}
\begin{split}
&\frac{4(p+1)}{p}\int_0^\infty\int_{\Omega}\Bigl|\nabla z^\frac{1}{2}+\frac{2k+ p(p+1)\frac{u }{u+1}}{2\sqrt{p(p+1)}}z^\frac{1}{2} \nabla w \Bigr|^2\varphi\\
&\ \ +\int_0^\infty\int_{\Omega}\frac{4k^2-p(p+1)^2\frac{u^2}{(u+1)^2}}{4(p+1)}
z|\nabla w|^2\varphi \\
&\leq  \liminf_{\varepsilon=\varepsilon_j\searrow0}\Bigl\{\int_\Omega z_\varepsilon(\cdot, 0)\varphi(\cdot, 0)+\int_0^\infty\int_{\Omega}  z_\varepsilon \varphi_t -2\int_0^\infty\int_{\Omega}z_\varepsilon^\frac{1}{2}\nabla z_\varepsilon^\frac{1}{2}\cdot\nabla\varphi\\
&\ \   -p\int_0^\infty\int_{\Omega}\frac{u_\varepsilon z_\varepsilon}{u_\varepsilon+1} \nabla w_\varepsilon\cdot\nabla\varphi-p\int_0^\infty\int_{\Omega}\frac{u_\varepsilon z_\varepsilon}{u_\varepsilon+1} (1-u_\varepsilon^{\theta-1}- v_\varepsilon)\varphi\\
 & \ \ +k\int_0^\infty\int_{\Omega}w_{\varepsilon}z_\varepsilon\varphi-k\int_0^\infty\int_{\Omega}\frac{( u_{\varepsilon}+ v_{\varepsilon})z_\varepsilon}{1+\varepsilon( u_{\varepsilon}+ v_{\varepsilon})}\varphi\Bigr\}\\
 &=\int_\Omega z(\cdot, 0)\varphi(\cdot, 0)+\int_0^\infty\int_{\Omega}  z \varphi_t -2\int_0^\infty\int_{\Omega}z^\frac{1}{2}\nabla z^\frac{1}{2}\cdot\nabla\varphi\\
&\ \   -p\int_0^\infty\int_{\Omega}\frac{u z}{u+1} \nabla w\cdot\nabla\varphi-p\int_0^\infty\int_{\Omega}\frac{u z}{u+1} (1-u^{\theta-1}- v)\varphi\\
 & \ \ +k\int_0^\infty\int_{\Omega}w_{\varepsilon}z\varphi-k\int_0^\infty\int_{\Omega} ( u+ v)z \varphi
\end{split}
\end{equation*}
With $\phi,\Phi$ and $\xi$ being as in \eqref{3.17}, from the specialization of \eqref{2.16}  in \eqref{3.23}, we establish  \eqref{2.9} for any such $\varphi$. Together Lemma \ref{lem3.5} with \eqref{3.48}--\eqref{3.57}, we can readily obtain \eqref{2.5}--\eqref{2.7}.
\end{proof}

\begin{proof}[Proof of Theorem \ref{Thm1.1}]
With $\phi,\Phi$ and $\xi$ defined  by  \eqref{3.17},  all ingredients in the  notion  of generalized solution to \eqref{1.1} in Definition \ref{def2.1} have been guaranteed by Lemmas \ref{lem3.9} and  \ref{lem3.11}, and so the  proof is complete.
\end{proof}

\textbf{Acknowledgments}  The authors would like to thank Prof.  Michael Winkler for sharing his preprint \cite{33}.  G. Ren was supported by the National Natural Science Foundation of China (No.12001214). T. Xiang was funded by the National Natural Science Foundation of China (Nos. 12071476  and 11871226) and  the Research Funds  of Renmin University of China (No. 2018030199).

\end{document}